\documentclass[12pt]{amsart}
\usepackage{tikz-cd}
\usepackage{pdfsync}
\usepackage[latin1]{inputenc}
  \usepackage{amsmath,amssymb}
  \usepackage{colortbl}
 \usepackage[english]{babel}
 \usepackage [english]{babel}
   \usepackage [T1] {fontenc}
\newtheorem{thm}{Theorem}[section]
\newtheorem{cor}[thm]{Corollary}
\newtheorem{lem}[thm]{Lemma}
\newtheorem{prop}[thm]{Proposition}
\theoremstyle{definition}
\newtheorem{defn}[thm]{Definition}
\newtheorem{rem}[thm]{Remark}
\newtheorem{ex}[thm]{Example}

\newtheorem{exm}[thm]{Example}
 \numberwithin{equation}{section}
\newcommand\A{\mathcal{A}}

\newcommand\E{\mathcal{E}}
\newcommand\F{\mathbb{F}}
\newcommand\N{\mathbb{N}}

\newcommand\R{{\mathbb R}}
\newcommand\C{\mathbb{C}}
\newcommand\K{\mathcal{K}}
\newcommand\B{\mathcal{B}}

\newcommand\Lm{\mathcal{L}}

\newcommand\la{\lambda}

\newcommand{\e}{\varepsilon}

\newcommand{\tl}{\widetilde{\lambda}}

\def\ds{\displaystyle}

\begin{document}
\title{Measurability, Spectral Densities and Hypertraces\\ in Noncommutative Geometry}
\author{Fabio E.G. Cipriani, Jean-Luc Sauvageot}
\address{(F.E.G.C.) Politecnico di Milano, Dipartimento di Matematica,
piazza Leonardo da Vinci 32, 20133 Milano, Italy.} \email{fabio.cipriani@polimi.it}
\address{Institut de Math\'ematiques, de Jussieu -- Paris Rive Gauche, CNRS-UniversitÃ© Denis Diderot, F-75205 Paris Cedex 13, France}
\email{jean-luc.sauvageot@imj-prg.fr}
\footnote{This work has been supported by Laboratoire Ypatia des Sciences Math\'ematiques C.N.R.S. France - Laboratorio Ypatia delle Scienze Matematiche I.N.D.A.M. Italy (LYSM).}
\date{November 30, 2021}
\subjclass{58B34, 47L20}
% ----------------------------------------------------------------
\begin{abstract}
We introduce, in the dual Macaev ideal of compact operators of a Hilbert space, the spectral weight $\rho(L)$ of a positive, self-adjoint operator $L$ having discrete spectrum away from zero. We provide criteria for its measurability and unitarity of its Dixmier traces ($\rho(L)$ is then called spectral density) in terms of the growth of the spectral multiplicities of $L$ or in terms of the asymptotic continuity of the eigenvalue counting function $N_L$. Existence of meromorphic extensions and residues of the $\zeta$-function $\zeta_L$ of a spectral density are provided under summability conditions on spectral multiplicities. The hypertrace property of the states $\Omega_L(\cdot)={\rm Tr\,}_\omega (\cdot\rho(L))$ on the norm closure of the Lipschitz algebra $\A_L$ follows if the relative multiplicities of $L$ vanish faster than its spectral gaps or if $N_L$ is asymptotically regular.

%\vskip1.0truecm
%\centerline{\includegraphics[width=3.0in,height=2.0in]{Reeb1}}
%\vskip0.2truecm
\end{abstract}
\maketitle

\section{Introduction}

%\subsection{}\-

Trace theorems for unbounded Fredholm modules $(\A,h,D)$, alias K-cycles or spectral triples,
%\cite{BJ},
subject to various summability behaviours, date back to the dawning of Noncommutative Geometry \cite{Co1}, \cite{Co2}, \cite{Co3}. They were proved under finite summability in \cite{Co2} and hold true also under summability in the dual Macaev ideal in \cite{CGS}.
%$\theta$-summability \cite{Co2}, subexponential summability \cite{V}?
They were used to deduce hyperfinitness of weak closure of the $^*$-algebra $\A$ in certain representations and to rule out the existence of unbounded Fredholm modules or quasidiagonal approximate units in normed ideals, with specific summability conditions (see \cite{Co2}, \cite{V}). Also, an hypertrace constructed by
$(\A,h,D)$ provides a Hilbert bimodule, unitary representation of the universal graded differential algebra $\Omega^*(\A)$ (\cite{Co3} Chapter 6.1 Proposition 5).\\
Here we associate a {\it spectral weight} $\rho(|D|)$ in the dual Macaev ideal $\Lm^{(1,\infty)}(h)$, to any unbounded Fredholm module $(\A,h,D)$ and, more in general, to any filtration $\mathcal{F}$ of a Hilbert space $h$ (in the sense of \cite{V}). The spectral weight $\rho(|D|)$ depends, in particular, on the spectral multiplicities of $D$ but not on the location of its eigenvalues.\\
Under an assumption of non exponential growth of the filtration, we show measurability of $\rho(|D|)$ and under the asymptotic continuity of the eigenvalue counting function $N_{|D|}$,  we prove also the unitarity ${\rm Tr\,}_\omega(\rho(|D|))=1$ of the Dixmier traces. In this situation $\rho(|D|)$ is called the {\it spectral density} of $D$ and one may deal with the {\it volume states} $\Omega_{|D|}(a):={\rm Tr}_\omega (a\cdot \rho(|D|))$ on the norm closure C$^*$-algebra $A$ of $\A$, provided by any fixed Dixmier ultrafilter $\omega$.\\
In commutative terms, i.e. dealing with the standard spectral triple $(C^\infty(M), L^2(Cl(M)), D)$ of a compact, closed,  Riemannian manifold, taking into account  multiplicities only and not the whole spectrum itself, one reconstructs the Riemann probability measure of $M$ loosing information about the volume $V(M)$ and the dimension $d(M)$. On the other hand, this has the advantage to dispense with summability hypotheses and, for example, to recover the unique trace on the reduced C$^*$-algebra $C^*(\Gamma)$ of a finitely generated, countable discrete group $\Gamma$, no matter its growth is. Also, using the density $\rho(|D|)$, one is able to treat, on the same foot, situations like Euclidean domains of infinite volume whose Dirichlet-Laplacian has discrete spectrum or certain hypoelliptic $\Psi$DO on compact manifolds, where the asymptotics of the spectrum of $D$ is not à la Weyl.\\
Under summability conditions on the spectral multiplicities of $|D|$, the $\zeta$-function $\zeta_{|D|}$ of the spectral density $\rho(|D|)$ is shown to be meromorphic on an half plane containing $z=1$ and that its residue is there unitary.\\
%Finally, we consider the volume states $\Omega_{|D|}$ showing that, under an asymptotic regularity on $N_{|D|}$, they are hypertraces on $A$.\\
%{\bf oppure: }
Finally, we show that the volume states $\Omega_{|D|}$ are hypertraces on $A$ provided $N_{|D|}\sim\varphi$ for a locally Lipschitz function 
$\varphi$ such that the essential limit of $\varphi'/\varphi$ vanishes at infinity. This condition is satisfied when the sequence of {\it relative multiplicities} of $|D|$ vanish faster than the sequence of its {\it spectral gaps}.\vskip0.2truecm\noindent
The work is organized as follows. In Section 2 we introduce the spectral weight $\rho (L)$ of a positive, self-adjoint operator $L$ having discrete spectrum away from zero. Its measurability is proved in terms of the sub-exponential growth of its spectral multiplicities, as a consequence of the asymptotic continuity of the counting function $N_L$ or as a byproduct of the nuclearity of the semigroup $e^{-tL}$. In Section 3 we prove existence of analytic extensions and residues of the $\zeta$-function $\zeta_L$ of a density $\rho(L)$, in terms of summability of the multiplicities of $L$. In Section 4, the volume states $\Omega_L(\cdot)={\rm Tr}_\omega(\cdot \rho(L))$ are introduced and in Section 6 we show that they are hypertraces on the Lipschitz algebra $\A_L$, under asymptotic smoothness of the counting function $N_L$ or when the relative multiplicities vanish faster than the spectral gaps of $L$. Section 5 is dedicated to various examples concerning i) k-cycles on compact manifolds given by (hypo)elliptic $\Psi$DO $L$ ii) k-cycles on the C$^*$-algebra $\mathcal{P}(M)$ of scalar, $0$-order $\Psi$DO, which are associated to scalar, $1$-order $\Psi$DO $L$ iii) k-cycles associated to multiplication operators on the group C$^*$-algebra of countable discrete groups iv) Dirichlet-Laplacians of Euclidean domains of infinite volume v) Kigami-Laplacians on selfmilar fractals vi) the Toeplitz C$^*$-algebra generated by an isometry and the canonical multiplication operator $L$ on natural $\N$ and prime numbers $\mathbb{P}$ vii) unbounded Fredholm modules built using Hilbert space filtrations. In the final subsection 6.5, the structure of the volume states $\Omega_L$ on C$^*$-algebras extensions is briefly outlined.

\section{Measurable densities associated to operators with discrete spectrum}
In this section we define the spectral weight of a nonnegative, self-adjoint operator $(L,D(L))$ on a Hilbert space $h$, having discrete spectrum away from zero and investigate its measurability in the framework of Connes' NCG. We always keep in mind the situation where $L=|D|$ for a spectral triple $(\A,D,h)$.
\subsection{Eigenvalue counting function and multiplicities}
In this section, we consider a densely defined, nonnegative, unbounded, self-adjoint operator $(L,D(L)$ on a Hilbert space $h$ with spectrum $sp(L)$ and spectral measure E$^L$.\\
Letting $P_0:=E^L(\{0\})$ be the orthogonal projection onto the kernel of $L$, we fix the following notations of functional calculus: by convention, for a measurable function $f\,:\,\R_+\to \R_+$, the operator $f(L)$ will be $0$ on the subspace $P_0(h)={\rm ker\,}L$ and $f\big(L(I-P_0)\big)$ on the subspace $(I-P_0)h=({\rm ker\,}L)^\perp$. For example, with this convention, for $s>0$, $L^{-s}$ is the nonnegative, densely defined operator defined as $0$ on $P_0(h)$ and $\big(L(I-P_0)\big)^{-s}$ on $(I-P_0)h$.
\vskip0.2truecm\noindent
In this section we shall suppose that $L$ has {\it discrete spectrum off of its kernel} in the sense that $sp(L)\setminus\{0\}$ is discrete: this is equivalent to say that $L^{-1}$ is a compact operator in $\B(h)$.
\vskip0.2truecm\noindent
We shall adopt two alternative ways for describing the spectrum, out of its kernel:\\
{\it First way}: $sp(L)\backslash \{0\}=\{0<\lambda_1(L)\leq \cdots \leq \lambda_n(L)\leq \cdots\}$
\vskip0.1truecm\noindent
where the eigenvalues $\la_n(L)$ are numbered increasingly with repetition according to their multiplicity.
\vskip0.1truecm\noindent
{\it Second way}: $sp(L)\backslash \{0\}=\{0<\tl_1(L)<\cdots <\tl_k(L)<\cdots\}$\\
where the distinct eigenvalues $\tl_k(L)$ are numbered increasingly.
\vskip0.2truecm\noindent
Since $L$ is assumed unbounded, we have $\lim_{n\to\infty}\lambda_n(L)=\lim_{k\to\infty}\tl_k(L)=+\infty$.

The {\it multiplicity} of the eigenvalue $\tl_k(L)$ is denoted by $m_k:={\rm Tr}(E^L(\{\tl_k\}))$ while the {\it cumulated multiplicity} is defined as $M_k:={\rm Tr}(E^L((0,\tl_k])$ so that $M_k=N_L(\tl_k(L))=\sum_{j=1}^k m_j$. By convention, $M_0:=0$. We will refer to the ratio $m_k/M_k$ as the {\it relative multiplicity} of the eigenvalue $\tl_k(L)$. The two labellings correspond through the relation:
\[
\la_n(L)=\tl_k(L),\qquad M_{k-1}<n\leq M_k.
\]
\begin{rem}
We shall adopt the simplified notations $\lambda_1,\cdots ,\lambda_n,\cdots$ and $\tl_1,\cdots, \tl_k,\cdots$ whenever no confusion can arise.
\end{rem}
The {\it eigenvalue counting function} $N_L:\R_+\to\mathbb{N}$ is defined as
\[
N_L(x):={\rm Tr}(E^L((0,x]))=\sharp\{n\in \N^*\,|\,\la_n(L)\leq x\}
\]
where $\rm Tr$ is the normal, semifinite trace on $B(h)$.
\vskip0.1truecm
Let us summarize some basic properties of the counting function
\begin{lem}\label{Fk}
i) $N_L$ is a non decreasing function, right continuous with left limits. For $x\in \R_+$, let us denote $N^-_L(x)=\lim_{\delta\downarrow 0} N_L(x-\delta)$ the left limit function of $N_L$.
\vskip0.2truecm\noindent
ii) $N_L(x)=M_k$ for $\tl_k(L) \leq x < \tl_{k+1}(L)$.
\vskip0.2truecm\noindent
iii) $\ds\limsup_{x\to +\infty}\frac{N_L(x)}{N^-_L(x)}=\limsup_{k\to \infty} \frac{M_k}{M_{k-1}}$\,.
\end{lem}
\begin{proof} Properties i) and ii) are obvious from the definition of $N_L$. For iii), it is enough to observe that, for $x\not\in sp(L)$, $N^-_L(x)=N_L(x)$ and that, for $x=\tl_k(L)$, $N_L(x)=M_k$ while $N^-_L(x)=M_{k-1}$.
\end{proof}

\subsection{Spectral weights}

\begin{defn}(Spectral weights)
The operator defined as
\[
\rho(L):=N_L(L)^{-1}
\]
will be called the {\it spectral weight of} $L$. As $sp(L)\setminus\{0\}$ is discrete and unbounded and $N_L$ is nondecreasing, it follows that $\rho(L)$ is nonnegative and compact.
\end{defn}

\begin{prop}\label{TrFLn}
i) The eigenvalues of the spectral weight are given by
\[
\mu_n(\rho(L))=N_L(\lambda_n(L))^{-1}=\frac{1}{M_k}\qquad {\it for}\qquad M_{k-1}<n\leq M_k\qquad {\it and}\qquad k\ge 1;
\]
ii) we also have the bounds
\begin{equation}\label{la}
\frac{M_{k-1}}{M_{k}}\cdot\frac{1}{n} < \mu_n(\rho(L))\leq \frac{1}{n}\qquad {\rm for}\quad M_{k-1}<n\leq M_k \qquad {\it and}\qquad k\ge 1.
\end{equation}

\end{prop}
{\it Proof.} For i), apply Lemma \ref{Fk}. For ii), notice that for $n$ as considered, we have $\la_n(L)=\tl_k(L)$ and $N_L(\la_n(L)) = N_L(\tl_k(L))=M_k$. From one side, $M_k\geq n$ and thus $N_L(\la_n(L))\geq n$. On the other side, $n>M_{k-1}$ and thus $\ds N_L(\la_n(L))^{-1}=M_k^{-1} > \frac{1}{n}\cdot\frac{M_{k-1}}{M_k}$\,.\hfill $\square$

%\begin{rem}
\subsection{Weights by filtrations}
For any Borel measurable, strictly increasing function $f:\R_+\to \R_+$ such that $\lim_{x\to \infty} f(x)=+\infty$, we have $\rho(f(L))=\rho(L)$.
In fact $\rho(L)$ depends on $L$ only through the Hilbert space filtration of spectral subspaces of $h$
\[
\{E^L((0,\tl_k(L)])\}_{k=1}^{+\infty}.
\]
In \cite{V} Proposition 5.1, D.V. Voiculescu, motivated by the existence of quasicentral approximate units relative to normed ideals, provided a general construction of spectral triples $(\A, h, D)$ on a C$^*$-algebra $A$, represented in a  Hilbert space $h$, associated to given filtrations $h_0\subset h_1\subset \cdots\subset h$.
He consider the filtration of $A$ given by $V_k:=\{T\in A: T(h_j)\cup T^*(h_j)\subseteq h_{j+k},\,\, \forall\,\, j\in\N\}$ for $k\in\N$, assuming that $\A:=\bigcup_{k\in\N}V_k$ is dense in $A$. Denoting by $P_j$ the projection onto $h_j$, the Dirac operator is defined as $D:=\sum_{j\ge 1}(I-P_j)$ (so that $D=|D|$). The spectrum of $D$ is $\N$ with cumulated multiplicities $M_j:={\rm dim\,}(h_j)$ and the spectral weight is given by
\[
\rho(D)=\sum_{k \ge 1}^{+\infty} \frac{1}{M_k}\cdot (P_{k+1}-P_k).
\]
%\end{rem}

\medskip \subsection{Dixmier traces.}\label{DixmierTraces} In this work we shall consider the non normal traces associated with ultrafilters on $\N$, introduced by J. Dixmier \cite{Dix}, making use, in particular, of the approach proposed in \cite{CPS}.

We shall start with a state $\omega$ on $L^\infty(\R_+^*)$ having all the properties of \cite{CPS} Theorem 1.5.

First, $\omega$ is a limit process at $+\infty$ in the sense that
\begin{equation*} ess\, \liminf_{t\to +\infty} f(t)\leq \omega(f) \leq ess\,\limsup_{t\to +\infty} f(t)\quad f\in L^\infty(\R_+^*)\, \tag{1}
\end{equation*}
so that we can write as well $\omega(f):=\omega-\lim_{t\to +\infty} f(t)$\,.

Then we require this limit process satisfies the following invariance properties :
\begin{itemize}
\item[(2)] $\ds\omega-\lim_{t\to +\infty} f(st)=\omega-\lim_{t\to +\infty} f(t)$, $f\in L^\infty(\R_+^*)$, $s\in \R_+^*$
\item[(3)] $\ds\omega-\lim_{t\to +\infty} f(t^s)=\omega-\lim_{t\to +\infty} f(t)$,  $f\in L^\infty(\R_+^*)$, $s\in \R_+^*$
\item[(4)] $\ds \omega-\lim_{t\to \infty} {\ds\frac{1}{\rm{Log}(t)}\int_1^t f(s)\frac{ds}{s}}=\omega-\lim_{t\to +\infty}  f(t)$,  $f\in L^\infty(\R_+^*)$\,.
\end{itemize}

To such $\omega$ are associated an ultrafilter on $\N$, still denoted $\omega$, by \\
\centerline{$\ds\omega-\lim_{n\to \infty} f(n)=\omega-\lim_{t\to +\infty} f([t])\quad f\in \ell^\infty(\N) \text{ with }[t]=\text{integer part of }t\,.$
}

The associated Dixmier Trace is defined on  the dual Macaev ideal of compact operators
\[
\Lm^{(1,\infty)}(h):=\{T\in \mathcal{K}(h):\sup_{N\ge 2}\frac{1}{{\rm Log} N}\cdot\sum_{n=1}^N\mu_n(|T|)<+\infty\}
\]
as
\begin{equation}\label{Tromega}
Tr_\omega(T)=\omega-\lim_{N\to \infty} \frac{1}{{\rm Log} N}\sum_{n=1}^N \mu_n(T)\qquad T\in \Lm^{(1,\infty)}(h)_+\, .
\end{equation}
The operator $T$ is said to be {\it measurable} if its Dixmier trace $Tr_\omega\big(\rho(L)\big)$ is independent upon the Dixmier ultrafilter $\omega$. Conditions ensuring measurability can be given in terms of Cesaro means (see Proposition 6 Chapter 4.2.$\beta$ in \cite{Co2}).

\medskip To $\omega$ on $L^\infty(\R_+^*)$ as above is associated an alternative
 limit process  $\widetilde \omega$ on $\R$ defined as \\
\centerline{$ \widetilde \omega -\lim_{t\to +\infty} f(t)=\omega-\lim_{t\to +\infty} f\big(\rm{Log}(t)\big)\quad f\in L^\infty(\R)\,.$}

The asymptotic behaviour of $\ds \frac{1}{{\rm Log} N}\sum_{n=1}^N \mu_n(T)$ and the limit behaviour of $\ds(s-1)Tr(T^s)$ as $s\to 1^+$ are related by the following equality [CPS, Theorem 3.1]
\begin{equation} \label{CPS31}
Tr_\omega(T)=\widetilde \omega-\lim_{r\to +\infty} \frac{1}{r}Tr\big(T^{1+\frac{1}{r}}\big)\qquad T\in \Lm^{(1,\infty)}(h)_+\, .
\end{equation}
Furthermore with $T$ as above and any $A\in \B(h)$, we have (see \cite{CPS} Theorem 3.8)
\begin{equation} \label{CPS38}
Tr_\omega(A\,T)=\widetilde \omega-\lim_{r\to +\infty} \frac{1}{r}Tr\big(A\,T^{1+\frac{1}{r}}\big)\qquad T\in \Lm^{(1,\infty)}(h)_+\, .
\end{equation}

\medskip \subsection{Asymptotics for the spectral weights.}\-
\begin{prop}\label{asymptotics}
i) The spectral weight belongs to the dual Macaev ideal $\Lm^{(1,\infty)}(h)$ and
\[
\frac{1}{c}\le Tr_\omega\big(\rho(L)\big)\leq 1\qquad {\it for\,\, all\,\, Dixmier\,\,ultrafilter\,\,}\omega
\]
where
\[
c:=\limsup_{x\to+\infty} \frac{N_L(x)}{N^-_L(x)}=\limsup_{k\to\infty}\frac{M_k}{M_{k-1}}.
\]
ii) $\rho(L)^s$ is trace class for all $s>1$ and $\limsup_{s\downarrow 1} (s-1)Tr(\rho(L)^s)\leq 1$. \\
iii) If
\[
\limsup_{x\to+\infty} \frac{N_L(x)}{N^-_L(x)}=\lim_{k\to\infty} \frac{M_{k}}{M_{k-1}}=1
\]
then we have
\begin{itemize}
\item[iii.a)] $\mu_n(\rho(L))\sim 1/n$ as $n\to \infty$.
\vskip0.1truecm\noindent
\item[iii.b)] $Tr_\omega(\rho(L))=1$ for all Dixmier ultrafilter $\omega$ and the spectral weight $\rho(L)$ is measurable.
\vskip0.1truecm\noindent
 \item[iii.c)] $\lim_{s\downarrow 1} (s-1)Tr(\rho(L)^s)=1$.
\end{itemize}
%\vskip0.2truecm\noindent
%iv) If
%\[
%c:=\limsup_{x\to+\infty} \frac{N_L(x)}{N^-_L(x)}=\limsup_{k\to\infty}\frac{M_k}{M_{k-1}}<+\infty,
%\]
%we have \par \centerline{$\ds Tr_\omega\big(\rho(L)\big)\geq \frac{1}{c}$\,.}
\end{prop}
\begin{proof}
i) follows from inequality $\mu_n(\rho(L))\leq 1/n$ of Proposition \ref{TrFLn} and from inequality $\liminf_{n\to \infty} \mu_n(\rho(L)) \geq 1/cn$ which follows by (\ref{la}). ii) follows again from inequality $\mu_n(\rho(L))\leq 1/n$ of Proposition \ref{TrFLn}. iii.a) comes from the double inequality (\ref{la}), while iii.b) and iii.c) are straightforward consequences of iii.a).
\end{proof}

\begin{defn}(Spectral densities)
The spectral weight $\rho(L)$ will be called {\it spectral density} provided it is measurable and ${\rm Tr}_\omega(\rho(L))=1$ for all Dixmier ultrafilter $\omega$.
\end{defn}

We may have measurability of a spectral weight even if it is not a density:

\begin{prop}\label{lim}
If $\ds \lim_{k\to\infty} \frac{M_{k+1}}{M_k}=c>1$, then the spectral weight $\rho(L)$ is measurable and
\[
Tr_\omega(\rho(L))= \lim_{s\downarrow 1}(s-1) Tr\big(\rho(L)^s\big)=
 \displaystyle\frac{c-1}{c\,\rm{Log}\, c}\qquad \text{for all Dixmier ultrafilter}\quad \omega\,.
\]
\end{prop}
\begin{proof}
 Fix $0<\e<c-1$ and $K\in\mathbb{N}$ such that $(c-\e)M_k\leq M_{k+1} \leq (c+\e)M_k$ for $k\geq K$. This implies $(c-\e)^{k-K}M_K\leq M_k \leq (c+\e)^{k-K}M_K$ for $k>K$. For $N>M_K$, let $k(N)$ be the integer such that $M_{k(N)}\leq N \leq M_{k(N)+1}$. One has
 \[
 \rm{Log\,} M_{k(N)}\leq\rm{Log\,} N \leq \rm{Log\,} (M_{k(N)+1})\leq \rm{Log\,} M_{k(N)}+\rm{Log\,} (c+\e)
 \]
 so that $\rm{Log\,} N=\rm{Log\,} M_{k(N)}+O(1)$ as $N\to +\infty$. Gathering those two results, we get
\begin{equation} \label{LogN}
k(N)\,\rm{Log}(c-\e)+O(1)\leq \rm{Log} N \leq k(N)\,\rm{Log}(c+\e)+O(1),\qquad N\to +\infty\,.
\end{equation}

Moreover
\begin{equation}\begin{split}
\sum_{n=1}^N N_L(\lambda_n)^{-1}&=\sum_{k=1}^{K-1}\sum_{n=M_{k-1}}^{M_k}N_L(\lambda_n)^{-1}+ \sum_{k=K}^{k(N)}\sum_{n=M_{k-1}+1}^{M_k}N_L(\lambda_n)^{-1}+\sum_{n=M_{k(N)+1}}^N N_L(\lambda_n)^{-1} \\
&=\sum_{k=1}^{K-1}\sum_{n=M_{k-1}}^{M_k}N_L(\lambda_n)^{-1}+ \sum_{k=K}^{k(n)}\frac{M_k-M_{k-1}}{M_k}+ \frac{N-M_{k(N)}}{M_{k(N)+1}}
\end{split}\end{equation}
For $k\geq K$, we have $\ds \frac{c+\e-1}{c+\e} \leq \frac{M_k-M_{k-1}}{M_k} \leq \frac{c-\e-1}{c-\e}$, which provides
\begin{equation}
k(N) \frac{c+\e-1}{c+\e} +O(1) \leq \sum_{n=1}^N N_L(\lambda_n)^{-1} \leq k(N) \leq \frac{c-\e-1}{c-\e}+O(1)\,,\;N\to +\infty.
\end{equation}
With (\ref{LogN}), this inequality implies
\begin{equation*}\begin{split}
\frac{1}{k(N)\,\rm{Log}(c+\e)+O(1)}&\Big(k(N) \frac{c+\e-1}{c+\e} +O(1) \,\Big) \\
\leq &\frac{1}{\rm{Log} N} \sum_{n=1}^N N_L(\lambda_n)^{-1} \\ \leq
&\frac{1}{k(N)\,\rm{Log}(c-\e)+O(1)}\Big(k(N) \frac{c-\e-1}{c-\e} +O(1) \,\Big)
\end{split}\end{equation*}
which provides the first equality $\ds Tr_\omega(\rho(L))=\frac{c-1}{c}$.

\medskip The second equality $\ds  \lim_{s\downarrow 1} Tr\big(\rho(L)^s\big)=
 \displaystyle\frac{c-1}{c\,\rm{Log} c}$ is obtained through a similar computation. Fix $\e>0$ and $K\in \N$ such that $k\geq K\Rightarrow c-\e \leq M_{k+1}/M_k \leq c+\e$, hence $(c-\e)^{k-K}M_K\leq M_k \leq (c+\e)^{k-K}M_K$ for $k>K$. We have also, fot $k\geq K+1$, $\ds c-\e \leq \frac{M_k}{M_{k-1}}=\frac{m_k}{M_{k-1}}+1\leq c+\e$, i.e.
 $\ds c-1-\e \leq \frac{m_k}{M_{k-1}}\leq c+\e$ which implies
 $$\frac{c-1-\e}{c+\e}\leq \frac{m_k}{M_k}\leq \frac{c-1+\e}{c-\e}\,.$$
 We compute now for $s >1$
 \begin{equation*} \begin{split}
 Tr(\rho(L)^s)=\sum_k m_k M_k^{-s}=\sum_{k=1}^{K} m_k M_k^{-s}+\sum_{k=K+1}^\infty m_kM_k^{-s}= \sum_{k=1}^{K} m_k M_k^{-s}+\sum_{k=K+1}^\infty \frac{m_k}{M_k}M_k^{-s+1}
 \end{split} \end{equation*}
 We have $\ds (s-1)\sum_{k=1}^{K} m_k M_k^{-s}\to 0$ as $s\to 1$, while
 \begin{equation*}\begin{split}
\sum_{k=K+1}^\infty \frac{m_k}{M_k}M_k^{-s+1}&\leq \frac{c-1+\e}{c-\e}(c+\e)^{(-s+1)(k-K)}M_K^{-s+1} \\
&=\frac{c-1+\e}{c-\e}(c+\e)^{(-s+1)(-K)}M_K^{-s+1}\frac{1}{1-(c+\e)^{1-s}}
 \end{split}\end{equation*}
 with $1-(c+\e)^{1-s}\sim (s-1)\rm{Log}(c+\e)$ as $s\downarrow 1$.\\
 We have proved $\ds \limsup_{s\downarrow 1} (s-1)Tr(\rho(L)^s)\leq \frac{c-1+\e}{c-\e}\frac{1}{\rm{Log}(c+\e)}$.
A similar computation provides $\ds \liminf_{s\downarrow 1} (s-1)Tr(\rho(L)^s)\geq \frac{c-1-\e}{c+\e}\frac{1}{\rm{Log}(c-\e)}$ and the result.
\end{proof}
The hypothesis of the above result is verified in discrete free groups (see Example \ref{exfreegroup} below).
\medskip
Here is another criterium for $\rho(L)$ having a nonzero Dixmier Trace.
\begin{prop} If $M_k\sim f(k)$ ($k\to+\infty$) with $f\in C^1((0,+\infty))$, then
$$\limsup_{k\to \infty} \frac{M_{k+1}}{M_k}\leq e^C\;\text{ with } C:=\limsup_{x\to +\infty} \frac{f'(x)}{f(x)}\,.$$
Hence $Tr_\omega(\rho(L))\geq e^{-C}$ for all Dixmier ultrafilters $\omega$.
\end{prop}
\begin{proof}
Fix $\e>0$ and choose $K_\e\ge 1$ such that for all $k\ge K_\e$ we have $\displaystyle \frac{1-\e}{1+\e}\,\frac{f(k+1)}{f(k)} \leq \frac{M_{k+1}}{M_k}\leq \frac{1+\e}{1-\e}\frac{f(k+1)}{f(k)}$. It follows that $\displaystyle \limsup_{k\to \infty} \frac{M_{k+1}}{M_k}= \limsup_{k\to \infty} \frac{f(k+1)}{f(k)}$. Then, setting $\displaystyle C=\limsup_{x\to +\infty} \frac{f'(x)}{f(x)}$, we have $\displaystyle (\rm{Log} f)'(x)= \frac{f'(x)}{f(x)}\leq C+\e$ for $x$ large enough so that\\
$\ds \rm{Log}(f(k+1))-\rm{Log}(f(k))\leq C+\e$, i.e.$\displaystyle \frac{f(k+1)}{f(k)}\leq e^{C+\e}$ for $k$ large enough too.
\end{proof}

\subsection{Asymptotic continuity of the eigenvalue counting function and measurability.}
Here we link the measurability of the spectral weight to the asymptotic continuity of the eigenvalue counting function and to the asymptotic vanishing of the relative multiplicity.
\begin{prop}\label{continuity}
\vskip0.2truecm\noindent
i) If the counting function $N_L$ is aymptotically continuous in the sense that there exists a continuous function $\varphi :\R_+\to \R_+$ such that
\[
N_L(x)\sim \varphi(x)\,,\;x\to +\infty\, ,
\]
then
\[
\lim_{x\to +\infty} \frac{N_L(x)}{N^-_L(x)}= \lim_{k\to \infty } \frac{M_{k+1}}{M_k}=1\quad\text{or, equivalently,}\quad  \lim_{k\to \infty } \frac{m_{k}}{M_k}=0.
\]
ii) Conversely, if $\ds \lim_{x\to +\infty} \frac{N_L(x)}{N^-_L(x)}=1$ or $\ds \lim_{k\to \infty } \frac{M_{k+1}}{M_k}=1$ or $ \ds \lim_{k\to \infty } \frac{m_{k}}{M_k}=0$, then
 $N_L$ is asymptotically continuous.
\vskip0.1truecm\noindent
In both cases, $\rho(L)$ is a density and the properties  iii) of Proposition \ref{asymptotics} hold true. \\

\end{prop}
\begin{proof}
i) For $\e>0$ and $x\in \R_+$ large enough, we have
\[
(1-\e)\varphi(x)\leq N_L(x)\leq (1+\e)\varphi(x).
\]
As $\varphi$ is continuous, we have as well $(1-\e)\varphi(x)\leq N^-_L(x)$ so that $\ds \frac{N_L(x)}{N^-_L(x)} \leq \frac{1+\e}{1-\e}$.
This implies $\ds \limsup_{x\to +\infty} \frac{N_L(x)}{N^-_L(x)} \leq \frac{1+\e}{1-\e}$ for all $\e>0$, and finally $\ds\limsup_{x\to +\infty} \frac{N_L(x)}{N^-_L(x)}=1$. Lemma \ref{Fk} and Proposition \ref{TrFLn} provides the result. \\
 ii) Choose $\varphi$ continuous, piecewise affine such that $\varphi(\tl_k)=M_k$. This means
\[
\varphi(x)=M_{k-1}+t\big(M_k-M_{k-1})
\]
for $ x\in [\,\widetilde \lambda_{k-1}, \widetilde \lambda_k\,]$ and $x= \widetilde \lambda_{k-1}+t \big(\widetilde \lambda_k-\widetilde \lambda_{k-1}\big)$, $0\leq t\leq 1$. For such $x$ we have:
\[
\frac{\varphi(x)}{N_L(x)}=\frac{M_{k-1}}{M_k}+t\left(1-\frac{M_{k-1}}{M_k}\right)\to 1,\qquad x\to +\infty\,.
\]
\end{proof}
The condition $1=\limsup_{k}\frac{M_{k+1}}{M_k}(=\lim_{k}\frac{M_{k+1}}{M_k})$ ensuring measurability
%\[
%\limsup_{k\to+\infty}\frac{M_{k+1}}{M_k}=\lim_{k\to+\infty}\frac{M_{k+1}}{M_k}=1,
%\]
concerns the sub-exponential growth of the spectral multiplicity of $L$. It is in general not weaker than the condition $\limsup_{k}\sqrt[k]{M_k}=1$ since $\liminf_{k}\frac{M_{k+1}}{M_k}\le \liminf_{k}\sqrt[k]{M_k}\le \limsup_{k}\sqrt[k]{M_k}\le \limsup_{k}\frac{M_{k+1}}{M_k}$.
\vskip0.2truecm\noindent
Combining this results with the Karamata-Tauberian Theorem (in Appendix), we get a criterion of asymptotic continuity of $N_L$ in terms of regularity of the partition function $Z_L$.
\begin{prop}\label{karamata}
Suppose the contraction semigroup $\{e^{-tL}:t\ge 0\}$ to be nuclear
\[
Z_L(\beta):={\rm Tr}(e^{-\beta L})<+\infty\qquad {\rm for\,\,all\,\,}\beta >0
\]
%(i.e. $(L,D(L))$ has sub-exponential spectral growth rate)
and assume the partition function $Z_L$ to be regularly varying. Then for some $c>0$ we have
\[
N_L(x)\sim c\cdot Z_L(1/x)\qquad x\to+\infty.
\]
In particular, under these assumptions, $N_L$ is asymptotically continuous and $\rho(L)$ is a density.
\end{prop}
\begin{proof}
Apply Karamata-Tauberian Theorem to the measure $\mu:={\rm Tr\,}\circ E^L$ and then apply Proposition \ref{continuity}.
\end{proof}
The result may applied to  $\theta$-{\it summable spectral triples} $(\A,h,D)$, where $Z_{D^2}(\beta)={\rm Tr\,}(e^{-\beta D^2})<+\infty$ for all $\beta>0$, provided the partition function $Z_{D^2}$ is regularly varying, as a consequence of the identity $N_L(x)=N_{D^2}(x^2)$ for all $x>0$.
\begin{rem} (Physical interpretation of nuclearity and regularity)
In applications $L$ may represent the Hamiltonian of a quantum system. The nuclearity assumption on the semigroup $\{e^{-\beta L}:\beta >0\}$ is easily seen to be equivalent to the requirement that the mean value of the energy in the Gibbs equilibrium state is finite and non vanishing at any temperature
\[
\langle L\rangle_\beta=-\frac{{\dot Z}_L(\beta)}{Z_L(\beta)}=\frac{{\rm Tr}(L e^{-\beta L})}{{\rm Tr}(e^{-\beta L})}\qquad \beta>0.
\]
The hypothesis that $Z_L$ is regularly varying requires that for some $\gamma\in\mathbb{R}$
\[
\lim_{\beta\to 0^+}\frac{Z_L(s\beta)}{Z_L(\beta)}=s^\gamma\qquad s>0.
\]
If $\dot Z_L$ is regularly varying, say $\dot Z_L(s\beta)\sim s^{\gamma-1}\dot Z_L(\beta)$ for some $\gamma\in\R$ as $\beta\to 0^+$, by de l'Hospital theorem, also $Z_L$ is regularly varying $\lim_{\beta\to 0^+}\frac{Z_L(s\beta)}{Z_L(\beta)}=s\lim_{\beta\to 0^+}\frac{\dot Z_L(s\beta)}{\dot Z_L(\beta)}=s^\gamma$ and the mean energy $\langle L\rangle$ is regularly varying too with
\[
\langle L\rangle_{s\beta}=-\frac{{\dot Z}_L(s\beta)}{Z_L(s\beta)}\sim -\frac{s^{\gamma-1}\dot Z_L(\beta)}{s^\gamma Z_L(\beta)} = \frac{1}{s}\cdot \langle L\rangle_\beta,\qquad {\rm as\,\,}\beta\to 0^+,\quad {\rm for\,\,any\,\,fixed\,\,} s>0.
\]
\end{rem}

\section{Meromorphic extensions of zeta functions and residues}\-
The $\zeta$-function of $\rho(L)$ is defined as
\[
\zeta_L(s):=Tr(\rho(L)^s)=\sum_{n\ge 1} \mu_n(\rho(L))^s=\sum_{k\ge 1} m_k\cdot M_k^{-s}
\]
for all $s\in\C$ for which the series converges. Its domain and its analytic properties will be found by comparison with the Riemann $\zeta$-function
\[
\zeta_0(s)=\sum_{n\ge 1} n^{-s},
\]
which, initially defined on the half-plane $\{s\in\C:{\frak\,{\frak Re\,}}(s)>1\}$, is then extended analytically to $\C\setminus\{1\}$. Recall that $s=1$ is simple pole for $\zeta_0$ with unital residue.
\vskip0.2truecm\noindent
The following criteria for the asymptotic properties of the $\zeta$-function $\zeta_L(s)$ as $s\to 1$ are based on various growth rates of the spectral multiplicity $m_k$ as $k\to\infty$.

%{\color{red} Una propriet\`a preliminare pi\`u o meno ovvia : per $\eta\in (0,1)$ e $S>0$, esiste una constante $C$ tale che
%$\big|(1-(1-a)^{\sigma}\big|\leq C\,a$ per $0\leq a\leq 1-\eta$ e $s\in \C$, $|s|\leq S$.}

\subsection{Criteria for meromorphic extensions of the $\zeta$-functions $\zeta_L$}\-

\begin{lem}
For $\varepsilon\in [0,1)$ and $s\in\C$ such that ${\frak{Re}(s)\ge0}$, we have
\[
|1-(1-\varepsilon)^s|\le |s|\rm{Log}\big((1-\varepsilon)^{-1}\big)\, .
\]
\end{lem}
\begin{proof}
Setting $b:=\rm{Log}(1-\varepsilon)^{-1}$, $x:={\frak{Re\,}(s)\ge0}$, $f(t):=(1-\varepsilon)^{ts}=e^{-bst}$ for $t\in [0,1]$, we have
\[
\begin{split}
|1-(1-\varepsilon)^s|&=|f(1)-f(0)|=\Bigl|\int_0^1 f'(t)\,dt\Bigr|=\Bigl|-bs\int_0^1 e^{-bst}\,dt\Bigr|\le b|s|\cdot\int_0^1 |e^{-bst}|\, dt\le b|s|.
\end{split}
\]
\end{proof}

\begin{prop}\label{residue0} i) If $N_L$ is asymptotically continuous then the $\zeta$-function $\zeta_L$ is well defined on the half-plane $\{s\in\C:{\frak\,{\frak Re\,}}(s)>1\}$ and it admits the limit
\begin{equation}
\lim_{s\in \R\,,\;s\downarrow 1}(s-1)\,Tr(\rho(L)^s)=1\,.
\end{equation}
ii) If $\ds \sum_k \frac{m_k^2}{M_k^2}<+\infty$, then $\zeta_L$ is analytic on $\{s\in\C:{\frak\,{\frak Re\,}}(s)>1\}$ and it admits the limit
\begin{equation}
\lim_{s\in \C,\,\,{\frak Re}(s)>1,\;s\to 1}(s-1)\,Tr(\rho(L)^s)=1\,.
\end{equation}
\vskip0.1truecm\noindent
iii) If \;$\ds \sum_k \frac{m_k^2}{M_k^{1+\alpha}}<+\infty$ for some $\alpha \in (0,1)$, then $\zeta_L$ extends to a meromorphic function on the half-plane $\{s\in\C:{\frak Re}(s)> \alpha\}$ with a simple pole at $s=1$ and unital residue.
\end{prop}
\begin{proof} i) By Proposition \ref{asymptotics} iii.a), the asymptotic continuity of $N_L$ implies that $\mu_n(\rho(L))\sim 1/n$ as $n\to\infty$ so that $\zeta_L$ is well defined on $\{s\in\C:{\frak\,{\frak Re\,}}(s)>1\}$. The limit behaviour is just the content of Proposition \ref{asymptotics} iii.c).

ii) Notice first that the assumption implies $\ds \lim_{k\to \infty} \frac{m_k}{M_k}=0$. Hence $\ds \lim_{k\to \infty} \frac{M_{k-1}}{M_k}=1$, $N_L$ is asymptotically continuous by Proposition \ref{continuity} ii) and $\mu_n(\rho(L))\sim 1/n$ as $n\to\infty$ again by Proposition \ref{asymptotics} iii.a). Let us write $\delta_n :=1/n-\mu_n(\rho(L))$. According to (\ref{la}),  whenever $k\ge 1$ and $M_{k-1} < n \leq M_k$ we have
\[
0\leq \delta_n < \frac{1}{n}\left(1-\frac{M_{k-1}}{M_k}\right)= \frac{1}{n}\,\frac{m_k}{M_k}\;\text{ and } \; 0\leq n\delta_n< \frac{m_k}{M_k}<1.
\]
Let us estimate the difference
\begin{equation}\label{sum}
\zeta_0(s)-\zeta_L(s)=\sum_{n\ge 1} \big(n^{-s}-(n^{-1}-\delta_n)^s\big)=\sum_{n\ge 1} n^{-s} \big(1-(1-n\delta_n)^s\big)
\end{equation}
on the closed half-plane $\{s\in\C:{\frak Re}(s)>1\}$. Since $1-n\delta_n\ge M_{k-1}/M_k>0$ for $k\ge 2$, $M_0=0$, $M_1=m_1\ge 1$, by the above lemma, for $s\in\C$ with ${\frak Re\,}(s)\ge 1$, we have
\[
\begin{split}
|\zeta_0(s)-\zeta_L(s)|&\le\sum_{n\ge 1} |n^{-s}|\cdot \big|1-(1-n\delta_n)^s\big|\\
&=\sum_{n\ge 1} n^{-{\frak Re\,}(s)}\cdot \big|1-(1-n\delta_n)^s\big|\\
%&=\sum_{k\ge 1}\sum_{n>M_{k-1}}^{M_k} n^{-{\frak Re\,}(s)}\cdot \big|1-(1-n\delta_n)^s\big|\\
%&=\sum_{n>M_0}^{M_1} n^{-{\frak Re\,}(s)}\cdot \big|1-(1-n\delta_n)^s\big|
%   + \sum_{k\ge 2}\sum_{n>M_{k-1}}^{M_k} n^{-{\frak Re\,}(s)}\cdot \big|1-(1-n\delta_n)^s\big|\\
%&=|1-(1-\delta_1)^s|+\sum_{n\ge 2} |n^{-s}|\cdot \big|1-(1-n\delta_n)^s\big|\\
&\le|s|\cdot\sum_{n\ge 1} n^{-{\frak Re\,}(s)}\cdot \rm{Log} (1-n\delta_n)^{-1}\\
&\le|s|\cdot\sum_{k\ge 1}\sum_{n>M_{k-1}}^{M_k} n^{-{\frak Re\,}(s)}\cdot  \rm{Log} (1-n\delta_n)^{-1}\\
&\le|s|\cdot\sum_{k\ge 2}\sum_{n>M_{k-1}}^{M_k} n^{-{\frak Re\,}(s)}\cdot \rm{Log} M_k/M_{k-1} + |s|\cdot\sum_{n>M_{0}}^{M_1} n^{-{\frak Re\,}(s)}\cdot \rm{Log} (1-n\delta_n)^{-1}\\
&=|s|\cdot\sum_{k\ge 2}\sum_{n>M_{k-1}}^{M_k} n^{-{\frak Re\,}(s)}\cdot  \rm{Log} (1+m_k/M_{k-1}) + |s|\cdot\sum_{n\ge 1}^{m_1} n^{-{\frak Re\,}(s)}\cdot \rm{Log} (1-n\delta_n)^{-1}\\
&\le|s|\cdot\sum_{k\ge 2}\sum_{n>M_{k-1}}^{M_k} M_{k-1}^{-{\frak Re\,}(s)}\cdot m_k/M_{k-1} + |s|\cdot\sum_{n\ge 1}^{m_1} \rm{Log} (1-n\delta_n)^{-1}\\
%&\le|s|\cdot\sum_{k\ge 2}\sum_{n>M_{k-1}}^{M_k} n^{-{\frak Re\,}(s)}\cdot \rm{Log} M_k/M_{k-1} + |s|\cdot\sum_{n\ge 1}^{m_1} n^{-{\frak Re\,}(s)}\cdot \rm{Log} (1-n\delta_n)^{-1}\\
%&<|s|\cdot\sum_{k\ge 2}\sum_{n>M_{k-1}}^{M_k} M_{k-1}^{-{\frak Re\,}(s)}\cdot \rm{Log} (1+m_k/M_{k-1})+ |s|\cdot\sum_{n\ge 1}^{m_1} \rm{Log} (1-n\delta_n)^{-1}\\
%&\le|s|\cdot\sum_{k\ge 1}\sum_{n>M_{k-1}}^{M_k} M_{k-1}^{-{\frak Re\,}(s)}\cdot m_k/M_{k-1}\\
&=|s|\cdot\Bigl(\sum_{k\ge 2} m^2_k/M_{k-1}^{1+{\frak Re\,}(s)} + \sum_{n\ge 1}^{m_1} \rm{Log} (1-n\delta_n)^{-1}\Bigr)\\
&\le|s|\cdot\Bigl(\sum_{k\ge 2} m^2_k/M_{k-1}^{2} + \sum_{n\ge 1}^{m_1} \rm{Log} (1-n\delta_n)^{-1}\Bigr)<+\infty.\\
\end{split}
\]
where, under the current hypothesis, the series converge as $k\to \infty$ (since $M_k\sim M_{k-1}$). Then $\lim_{s\in \C\,,\;Re(s)>1\,,\;s\to 1}(s-1)[\zeta_0(s)-\zeta_L(s)]=0$ and the thesis follows since $\zeta_0$ has residue one at the simple pole $s=1$.
\vskip0.2truecm
iii) Fix $r>0$ and set $D_{r,\alpha}:=\{s\in\C: {\frak Re\,}(s)>\alpha,\,\, |s|<r\}$. Reasoning as above we have
\[
\begin{split}
\sum_{n\ge 1}\sup_{D_{r,\alpha}} \big|n^{-s}-(n^{-1}-\delta_n)^s\big|&\le \sup_{D_{r,\alpha}} |s|\cdot\Bigl(\sum_{k\ge 2} m^2_k/M_{k-1}^{1+{\frak Re\,}(s)} + \sum_{n\ge 1}^{m_1} \rm{Log} (1-n\delta_n)^{-1}\Bigr)\\
&\le r\cdot\Bigl(\sum_{k\ge 2} m^2_k/M_{k-1}^{1+\alpha} + \sum_{n\ge 1}^{m_1} \rm{Log} (1-n\delta_n)^{-1}\Bigr)\\
\end{split}
\]
so that the series of analytic functions $\sum_n \big(n^{-s}-(n^{-1}-\delta_n)^s\big)$ converges uniformly on $D_{r,\alpha}$ and then $\zeta_0(s)-\zeta_L(s)$ extends analytically on that domain. The thesis follows again by the above mentioned properties of $\zeta_0$.
\hfill $\square$\\

\begin{prop} \label{mMalpha} If \;$\ds m_k=O(M_k^\alpha)$, $k\to\infty$, for some $\alpha \in (0,1)$, then $\zeta_L$ extends to a meromorphic function on the half-plane $\{s\in\C:{\frak Re}(s)> \alpha\}$ with a simple pole at $s=1$ and unital residue.
\end{prop}
\begin{proof}
The assumption $\ds m_k=O(M_k^\alpha)$ implies that $m_k/M_k=(m_k/M_k^\alpha)M_k^{\alpha-1}\to 0$ as $k\to\infty$, so that, by Proposition \ref{continuity} ii), $N_L$ is asymptotically continuous. Applying Proposition \ref{asymptotics} iii.a) we have $\mu_n(L)\sim1/n$ as $n\to\infty$ and then, for any fixed $\beta\in (\alpha,1)$, the following series converge
\begin{equation*}
\sum_{k\ge 1} \frac{m_k^2}{M_k^{1+\beta}} \leq C \sum_{k\ge 1} \frac{m_k }{M_k^ {1+\beta-\alpha}}
=C\sum_{k\ge 1} \sum_{n=M_{k-1}+1}^{M_k}\;\mu_n(L)^{-(1+\beta-\alpha)}=C\sum_{n\ge 1}\mu_n(L)^{-(1+\beta-\alpha)}<+\infty\,.
\end{equation*}
Apply Proposition \ref{residue0} to conclude.
\end{proof}
The hypothesis of the previous result can be restated in terms of a bound on the error term in the asymptotically continuous behaviour of the counting function.
\begin{prop}\label{mMphi} For $\alpha\in (0,1)$, the two statements are equivalent
\vskip0.2truecm\noindent
(i) $m_k=O(M_k^\alpha)$, $\quad k\to \infty$
\vskip0.2truecm\noindent
(ii) There exists a continuous function $\varphi$ such that
$$N_L(x)=\varphi(x)+O(\varphi(x)^\alpha)\,,\qquad x\to +\infty\,.$$
\end{prop}
\begin{proof}
Suppose $m_k=O(M_k^\alpha)$. Notice first that $m_k=o(M_k)$ so that $M_k\sim M_{k+1}$ and $N_L$ is asymptotically continuous. Let $\varphi$ be the continuous piecewise affine function defined as
$$\varphi(x)=M_k+t(M_{k+1}-M_k)=M_k+t\,m_{k+1}\,,\;x\in [\tl_k,\tl_{k+1}]\,,\;x=\tl_k+t(\tl_{k+1}-\tl_k)\,\quad t\in [0,1].$$
One has $0\leq \varphi(x)-N_L(x)\leq m_{k+1}=O(M_{k+1}^\alpha)=O(\varphi(x)^\alpha)$, $x\to +\infty$, which provides $N_L(x)=\varphi(x)+O(\varphi(x)^\alpha)$.
\medskip
Conversely, if $\varphi$ is continuous and such that $N_L(x)=\varphi(x)+O(\varphi(x)^\alpha)$ as $x\to +\infty$, let us define the function $r(x)$ by the formula
$$N_L(x)=\varphi(x)(1+r(x))\,.$$
On one side, one has $r(x)=O(\varphi(x)^{\alpha-1}$ as $x\to +\infty$. On the other side, the function $r(\cdot)$ is right continuous with left limits $r^-(x):=\lim_{\delta \downarrow 0} r(x-\delta)$, in such a way that
$$N_L^-(x)=\varphi(x)(1+r^-(x))\,.$$
%Notice that we have $r^-(x)\leq r(x)=O(\varphi(x)^{\alpha-1})$.

\smallskip Fixing $\delta>0$ small enough we have $N_L(\tl_{k+1}-\delta)=M_k$ while $N_L(\tl_{k+1})=M_{k+1}$, i.e.\\ $M_k=\varphi(\tl_{k+1}-\delta)\big(1+r(\tl_{k+1}-\delta)\big)$  while $M_{k+1}=\varphi(\tl_{k+1})\big(1+r(\tl_{k+1})\big)$. Taking the quotient and making $\delta$ tend to $0$, we get
$$\frac{M_{k+1}}{M_k}=\frac{1+r(\tl_{k+1})}{1+r^-(\tl_{k+1})}=\frac{1+O(\varphi(x)^{\alpha-1})}{1+O(\varphi(x)^{\alpha-1})}=1+O(\varphi(x)^{\alpha-1})=1+O(M_{k}^{\alpha-1})\,
$$
which means  $\ds \frac{m_{k+1}}{M_k}=O(M_{k}^{\alpha-1})$ and $m_{k+1}=O(M_k^\alpha)=O(M_{k+1}^\alpha)$ (since $M_{k+1}\sim M_k$)\,.
\end{proof}
Summarizing, we have the following criterion of meromorphic extension.
\begin{thm}\label{merocor}
Suppose that there exists $\alpha\in (0,1)$ and a continuous function $\varphi$ such that
$$N_L(x)=\varphi(x)+O(\varphi(x)^\alpha)\,,\quad x\to +\infty\,.$$
Then the $\zeta$-function $\zeta_L(s)=Tr\big(\rho(L)^s\big)$ extends as a meromorphic function on the half plane $\{s\in\C:{\frak Re}(s)> \alpha\}$ with a simple pole at $s=1$ and unital residue.
\end{thm}
\begin{proof}
Apply Propositions \ref{mMphi} and \ref{mMalpha}\,.
\end{proof}

\subsection{Examples of densities and their $\zeta$-functions}\-

\begin{ex}(Compact smooth manifolds I)
Let $M$ be a compact, n-dimensional, orientable, smooth manifold without boundary with cotangent bundle $\pi_*:T^*M\to M$. Let $m^*$ be the symplectic volume measure on $T^*M$ and fix a smooth volume measure $m$ on $M$. Disintegrating $m^*$ with respect to $m$ by $\pi^*$, one gets a family of measures $m^*_x$ on $T^*_xM$ for $m$-a.e. $x\in M$ such that $m^*=\int_M m^*_x\cdot m(dx)$.
\par\noindent
Let $L$ be the Friedrichs extension of an $m$-symmetric, positive, elliptic, smooth pseudo differential operator of order $k\ge 1$ with classical symbol, defined on $C^\infty(M)\subset L^2(M,m)$. Let $p$ the principal symbol of $L$, understood as a real, homogeneous polynomial of degree $k$ on $T^*M$ or as a function on the cosphere bundle $S^*M$.  The spectrum of $L$ is discrete and the Weyl asymptotic formula for the eigenvalue counting function of $L$ reads
\begin{equation}\label{Weyl}
N_L(x)\sim c\cdot x^{n/k}\qquad x\to+\infty, \qquad  c:=\frac{1}{(2\pi)^n}\cdot \int_Mm^*_x \{p(x,\cdot)<1\}\cdot m(dx).
\end{equation}
It follows by Proposition \ref{continuity} that $N_L$ is asymptotically continuous and that $\rho(L)$ is a density. The H\"ormander estimates for the remainder term [Hor1] reads
\[
N_L(x)=c\cdot x^{n/k}+ O(x^{(n-1)/k})\qquad x\to+\infty
\]
and, by Theorem \ref{merocor}, it follows that the  $\zeta$-function $\zeta_L$ of the density $\rho(L)$ is meromorphic on $\{s\in\C:{\frak Re}(s)> 1-1/n\}$ and it has a simple pole in $s=1$ with unital residue. The order of the remainder term cannot be improved in general, but in case of a product of 2-spheres $M=S^2\times S^2$ and for the Laplace-Beltrami operator $L$, M. Taylor [Tay] proved the estimate
\[
N_L(x)=c\cdot {\rm Vol}(S^2\times S^2)\cdot x^2+O(x^{4/3})\qquad x\to+\infty,\qquad c:=(4\pi)^{-2}/\Gamma(3)
\]
which ensures that $\zeta_L$ is meromorphic on $\{s\in\C:\frak{Re}(s)>2/3\}$. As another example, one can consider the Laplace-Beltrami operator $L$ on the flat torus $\mathbb{T}^n=\mathbb{R}^n/\mathbb{Z}^n$ where the reminder term is $O(x^{(n-1)/2-\gamma})$ for some $\gamma>0$. The case $n=2$ corresponds to the classical Gauss problem: in fact $N_L(x)$ coincides with the number of points in $\mathbb{Z}^2$ falling within the circle of radius $x>0$. It is known [deH page 6] that $N_L(x)=\pi x+O(x^{\alpha})$ as $x\to+\infty$ with $\alpha\in (1/4,12/37)$. Still in [Tay] it is shown that on $S^2$, the sub-elliptic operator $L=X_1^2+X_2^2$ given by the sum of squares of two vector fields $X_1$, $X_2$, generating rotations around orthogonal axes, has a counting function with a non-Weyl asymptotic behaviour
\[
N_L(x)=\frac{1}{2}x\cdot{\rm Log\,}(x)+O(x)\qquad x\to+\infty.
\]
Again by Proposition \ref{continuity}, $N_L$ is asymptotically continuous and that $\rho(L)$ is a density. The asymptotic behaviour of the counting function $N_L$ of hypoelliptic $\Psi$DO is studied in [MS].
\end{ex}
\end{proof}

One can have a meromorphic extension even if the counting function is not asymptotically continuous:
\begin{exm}\label{exfreegroup}
Let $\F_p$ be the free group with $p$ generators ($p\geq 2$), $\ell$ the length function on $\F_p$ and $L$ the multiplication operator by $\ell$ on the Hilbert space $\ell^2(\F_p)$ (see \cite{Haa}).

The spectrum of $\ell$ is $\N$: $\tl_k=k$ for $k\geq 1$ with multiplicities $m_k=2p(2p-1)^{k-1}$ and $\ds M_k=\frac{2p}{2p-2}\big((2p-1)^k-1\big)$. Hence $m_k/M_k\to (2p-2)/(2p-1)>0$ and, by Proposition \ref{continuity}, $N_L$ is not asymptotically continuous. However, since $\lim_{k\to\infty}M_{k+1}/M_k=2p-1>1$, Proposition \ref{lim} implies that the spectral weight $\rho(L)$ is measurable and that $\ds Tr_\omega(\rho(L))=\lim_{s\downarrow 1} (s-1)Tr(\rho(L)^s)= \frac{2p-2}{(2p-1)\rm{Log} (2p-1)}$. We show that this limit is indeed a residue\,:

\begin{prop}\label{GL} With the notations of the above example, we have that the zeta function $\zeta_L(s)=Tr(\rho(L)^s)$ extends as a meromorphic function on the half plane $\{s\in\C:\frak{Re}(s)>0\}$ with a simple pole at $s=1$ and residue
\[
{\rm Res}_{s=1}(\zeta_L)=\lim_{s\in \C\,,\;Re(s)>0\,,\; s\to 1} (s-1)Tr(\rho(L)^s)= \frac{2p-2}{(2p-1)\rm{Log} (2p-1)}.
\]
\end{prop}
{\it Proof.} Let us compute for ${\frak Re}(s)>1$
\begin{equation*} \begin{split}
Tr(\rho(L)^s)&=\sum_{k\ge 1} m_k M_k^{-s} \\
&=(2p)^{-s+1}\frac{(2p-2)^s}{2p-1}\sum_{k\ge1} (2p-1)^{k}\big((2p-1)^k-1\big)^{-s}\\
&=(2p)^{-s+1}\frac{(2p-2)^s}{2p-1}\sum_{k\ge1} (2p-1)^{k-ks}\big(1-(2p-1)^{-k}\big)^{-s}\\
&=\varphi(s) \big(Z_1(s)-Z_2(s)\big)
\end{split}\end{equation*}
with \\
 $\bullet$ $\ds \varphi(s)=(2p)^{-s+1}\frac{(2p-2)^s}{2p-1}$\,: $\varphi$ extends as an analytic function on the whole complex plane\,;\ its value at $s=1$ is $(2p-2)/(2p-1)$\,;\\
 $\bullet$ $\ds Z_1(s)=\sum_k (2p-1)^{k-ks}=\frac{(2p-1)^{1-s}}{1-(2p-1)^{1-s}}$\,: $Z_1$ extends as a meromorphic function on the whole complex plane with one pole at $s=1$ which is simple, with residue $\ds \frac{1}{\rm{Log}(2p-1)}$\,; \\
 $\bullet$ $\ds Z_2(s)=\sum_k (2p-1)^{k-ks}\big(1-(1-(2p-1)^{-k}\big)^s\big)$\,: $Z_2$ appears as a sum of analytic functions, each of them being bounded that way : fix $S>0$, then there exists a constant $C$ such that
 $$\big|1-(1-(2p-1)^{-k}\big)^s\big|\leq C\,(2p-1)^{-k}\,,\;k\geq 1\,,\; |s|\leq S $$
 hence
 $$ \Big|(2p-1)^{k-ks}\big(1-(1-(2p-1)^{-k}\big)^s\big)\Big| \leq C (2p-1)^{-ks}\,,\; k\geq 1\,,\;|s|\leq S\,.$$
 The series $\sum_k (2p-1)^{k-ks}\big(1-(1-(2p-1)^{-k}\big)^s\big)$ converges locally uniformly on the upper half plane ${\frak Re}(s)>0$ and its sum defines an analytic function on the half plane $\frak{Re}(s)>0$.

 Gathering those intermediate results, the Proposition is proved. \hfill $\square$

%$$1-(2p-1)^{1-s}=1-e^{(1-s)\rm{Log}(2p-1)}\sim (s-1)\rm{Log}(2p-1)\quad \text{ as } s\to 1\,.$$
%~\hfill $\square$
\end{exm}

\section{Volume forms associated to spectral weights}

\subsection{Volume forms.}

Fix a Dixmier ultrafilter $\omega$ on $\N$ as obtained in subsection \ref{DixmierTraces}.
\begin{defn} The {\it volume form} on $\B(h)$ associated to $L$ and $\omega$ will be the linear form $\Omega_L$
$$\B(h)\ni T \to \Omega_L(A)=Tr_\omega(T\,\rho(L))\,.$$
The restriction of $\Omega_L$ to a sub-$C^*$-algebra $A\subset \B(h)$ will be the called {\it volume form} on $A$.
\end{defn}
Volume forms satisfy some obvious properties:
\begin{prop}
\vskip0.1truecm\noindent
i) $\Omega_L$ is a positive, hence uniformly continuous linear form on $\B(h)$ with norm less than $Tr_\omega(\rho(L))\leq 1$.
\vskip0.2truecm\noindent
ii) $\Omega_L$ vanishes on the $C^*$-algebra $\K(h)$ of compact operators, hence defines a positive linear form on the Calkin algebra $\B(h)/\K(h)$\,.
\vskip0.2truecm\noindent
iii) $\ds \Omega_L(T)=\widetilde \omega- \lim_{r\to \infty} \frac{1}{r}\,Tr\big(T\,\rho(L)^{1+\frac{1}{r}}\big)$ (recall the definition of $\widetilde\omega$ above \ref{CPS31}).
\end{prop}

\begin{prop}\label{volumeform}
For any measurable function $\varphi :\R_+\to \R_+$ such that
\[
N_L(x)\sim \varphi(x)\qquad x\to +\infty,
\]
i)  the operator $\varphi(L)^{-1}$ belongs to the ideal $\Lm^{(1,\infty)}(h)$  and
\[
Tr_\omega\big(\rho(L)\big)=Tr_\omega\big(\varphi(L)^{-1}\big)\qquad {\it for\,\, all\,\, ultrafilter}\,\, \omega;
\]
ii) for any $T\in \B(h)$, the compact operators $T\rho(L)$, $T\varphi(L)^{-1}$ belong to the ideal $\Lm^{(1,\infty)}(h)$ and the volume form can be represented as
\[
\Omega_L(T):=Tr_\omega\big(\,T\,\rho(L)\big)=Tr_\omega\big(\,T\,\varphi(L)^{-1}\big).
\]
\end{prop}
\begin{proof}
i) One has $N_L(x)^{-1}\sim \varphi(x)^{-1}$, which we write $\varphi(x)^{-1}=N_L(x)^{-1}+g(x)N_L(x)^{-1}$ for some function $g$ with $\lim_{x\to \infty} g(x)=0$. We then have $\varphi(L)^{-1}=\rho(L)+g(L)\rho(L)$ with $g(L)$ compact. Hence $g(L)\rho(L)$ belong to the closure $\Lm_0^{(1,\infty)}(h)$ in $\Lm^{(1,\infty)}(h)$ of the ideal of finite rank operators, on which the Dixmier trace vanishes and the result follows.\\
ii) Notice that $\B(h)\ni T \to Tr_\omega\big(T\rho(L)\big)$ is a positive linear form, hence is a norm continuous functional on $\B(h)$. As it is obviously $0$ whenever $T$ has finite rank, it vanishes on the $C^*$-algebra $\K(h)$ of compact operators. Hence, we have $Tr_\omega\big(g(L)\rho(L)\big)=0$ and, for any $T\in \B(h)$ and $Tr_\omega\big(\,T\,g(L)\rho(L)\big)=0$.
\end{proof}

%--------------------------------------------------------------------
%
%\begin{prop}
%For any measurable function $\varphi :\R_+\to \R_+$ such that
%\[
%N_L(x)\sim \varphi(x)\qquad x\to +\infty,
%\]
%i) the operator $\varphi(L)^{-1}\in \Lm^{(1,\infty)}(h)$ is measurable and
%\[
%Tr_\omega\big(\rho(L)\big)=Tr_\omega\big(\varphi(L)^{-1}\big)\qquad {\it for\,\, all\,\, ultrafilter}\,\, \omega;
%\]
%ii) for any $T\in \B(h)$, the compact operators $T\rho(L)$, $T\varphi(L)^{-1}\in \Lm^{(1,\infty)}(h)$ are measurable and the residue state can be represented as
%\[
%\phi_L(T):=Tr_\omega\big(\,T\,\rho(L)\big)=Tr_\omega\big(\,T\,\varphi(L)^{-1}\big).
%\]
%\end{prop}
%\begin{proof}
%i) One has $N_L(x)^{-1}\sim \varphi(x)^{-1}$, which we write $\varphi(x)^{-1}=N_L(x)^{-1}+g(x)N_L(x)^{-1}$ for some function $g$ with $\lim_{x\to \infty} g(x)=0$. We then have $\varphi(L)^{-1}=\rho(L)+g(L)\rho(L)$ with $g(L)$ compact. Hence $g(L)\rho(L)$ belong to the closure $\Lm_0^{(1,\infty)}(h)$ in $\Lm^{(1,\infty)}(h)$ of the ideal of finite rank operators, on which the Dixmier trace vanishes and the result follows.\\
%ii) Notice that $\B(h)\ni T \to Tr_\omega\big(T\rho(L)\big)$ is a positive linear form, hence is a norm continuous functional on $\B(h)$. As it is obviously $0$ whenever $T$ has finite rank, it vanishes on the $C^*$-algebra $\K(h)$ of compact operators. Hence, we have $Tr_\omega\big(g(L)\rho(L)\big)=0$ and, for any $T\in \B(h)$ and $Tr_\omega\big(\,T\,g(L)\rho(L)\big)=0$.
%\end{proof}
%
%--------------------------------------------------------------------
%
Hereafter are the first examples of such volume states and linear forms.
\subsection{Compact smooth manifolds II}
i) In the framework and notations of Example 3.2.1, consider the action $g\mapsto M_g$ of the commutative C$^*$-algebra $C(M)$ by pointwise multiplication on $L^2(M,m)$.  By the Weyl asymptotic formula $N_L(x)\sim c\cdot x^{n/k}=:\varphi(x)$ as $x\to+\infty$, $N_L$ is asymptotically continuous, $\rho(L)$ is a density and the volume forms $\Omega_L$ are states. Since $\varphi(L)^{-1}=c^{-1}\cdot L^{-n/k}$, by Proposition \ref{volumeform} one has
$\Omega_L(T)=c^{-1}\cdot Tr_\omega\big(\,T\,L^{-n/k}\big)$. The restriction of these states to $C(M)$ are represented by probability measures $\nu_\omega$: $\Omega_L(M_g)=\int_M g\cdot d\nu_\omega$.
\vskip0.2truecm\noindent
Let $\pi^*:S^* M\to M$ be the cosphere bundle whose fibers are the rays of $T^*M$ and consider a {\it scalar valued}, elliptic, $m$-symmetric $1$-order $\Psi DO$ on $M$ with classical symbol, denoting by $D$ its self-adjoint extension to $L^2(M,m)$.

%t is known that $D$ has discrete spectrum and its modulus $|D|$ has an eigenvalue counting function such that $N_{|D|}(x)\sim c\cdot x^d$ for $x\to \infty$, some constant $c>0$ and $d:={\rm dim\,}(M)$.
\vskip0.2truecm\noindent
ii)  Since the operators $D$ and $M_g$, $g\in C^\infty(M)$, are scalar valued $\Psi DO$, their symbols commute. Also, since they are of order $1$ and $0$ respectively,  by the rules of pseudo differential calculus, the commutators $[D, M_g]$ are $0$-order $\Psi DO$, thus bounded. Hence $(C^\infty(M),D,L^2(M,m))$ is a $(d,\infty)$-summable spectral triple on $C(M)$, in the sense of A. Connes \cite{Co3}. The spectral weight $\rho(|D|)$ is a density and the volume forms $\Omega_{|D|}$ are states on $C(M)$ represented by probability measures $\nu_\omega$ on $M$: $\Omega_{|D|}(M_g)=\int_M g\cdot d\nu_\omega$.
\vskip0.2truecm\noindent
iii) Let $\A$ be the $^*$-algebra of {\it scalar valued}, $0$-order $\Psi DO$ on $M$, acting boundedly on $L^2(M,m)$. Let $\mathcal{P}(M)$ be the C$^*$-algebra of bounded operators on $L^2(M,m)$ generated by $\A$ (see \cite{HR}). Again, since $D$ and operators $T$ in $\A$ are scalar valued their symbols commute and since they are of order $1$ and $0$ respectively, the commutators $[T,D]$ have $0$-order and are thus bounded. Hence $(\A,D,L^2(M,m))$ is a $(d,\infty)$-summable spectral triple on $\mathcal{P(M)}$. Since
\[
  \begin{tikzcd}
    0 \arrow{r} & \mathcal{K}(L^2(M,m)) \arrow{r} & \mathcal{P}(M) \arrow{r}{\sigma} & C(S^*(M)) \arrow{r} & 0,
  \end{tikzcd}
\]
is a C$^*$-algebra extension (see \cite{D}, \cite{HR}) and the volume states $\Omega_{|D|}$ restricted to $\mathcal{P}(M)$ vanish on the ideal $\K(L^2(M,m))$, it follows that they
factorizes through suitable probability measures $\nu^*_\omega$ on the cosphere bundle $S^*(M)$. As the representation $g\mapsto M_g$ is injective, we can identify $C(M)$ with its image in $\B(L^2(M,m))$. Since $C(M)\cap \mathcal{K}(L^2(M,m))=\{0\}$, $C(M)\subset \mathcal{P}(M)$ and the restriction of the principal symbol map $\sigma$ is given by $\sigma(M_g)=g\circ\pi^*$ for any $g\in C^\infty(M)$, the measure
$\nu_\omega$ is the image of $\nu_\omega^*$ under $\pi^*:S^* M\to M$.
\vskip0.2truecm\noindent
%iii) As variations of the above situation, one may consider operators $D$ for which $D^2$ is a finite sum of H\"ormander vector fields. In this case $D$ is hypoelliptic and this guarantees that its spectrum is still discrete. The spectral exponent $d$ reduces to the homogeneous dimension of the given family of vector fields. Examples: Carnot groups.

\subsection{Multiplication operators on discrete groups}\label{mult} Let $G$ be a countable discrete group with unit $e$ and left regular representation $\lambda$ in the Hilbert space $l^2(G)$. If $\{\delta_g\}_{g\in G}$ is the canonical orthonormal base of $l^2(G)$, we have $\lambda(g)\delta_h=\delta_{gh}$, $g,h\in G$.

Let $\ell$ be a proper function from $G$ in $\R_+$ and $L$ the operator of multiplication by $\ell$ on $l^2(G)$\,: $L\delta_g=\ell(g)\delta_g$. The (discrete) spectrum of $L$ coincides with the image of $\ell$: $sp(L)=\{\ell(g):g\in G \}$. For any $e\neq g\in G$, $s\in \C$, $\frak{Re}(s)>1$, we have
\[
Tr(\lambda(g)\rho(L)^s) = \sum_h (\delta_h |\,N_L(\ell(g))^{-s}\delta_{gh})_{l^2(G)}=0.
\]

The generic element $a\in C^*_{red}(G)$ has a Fourier expansion $a=\sum_g a_g\lambda(g)$ and is a uniform limit of elements of which the Fourier expansion has finite support ($a_g=0$ except for a finite number of $g$'s). For $a$ with finite support, we have $Tr(a\rho(L)^s)=a_e Tr(\rho(L)^s)=\tau(a)Tr(\rho(L)^s)$, where $\tau$ is the canonical trace on $C^*_{red}(G)$: $\tau(a)=a_e$. By uniform continuity, this formula extends to any $a\in C^*_{red}(G)$.  Finally, applying formula (\ref{la}) of subsection \ref{DixmierTraces}, we get
$$\Omega_L(a)=Tr_\omega(a\rho(L))=\tau(a)Tr_\omega(\rho(L))\,, \; a\in C^*_{red}(G)\,.$$
Normalizing $\Omega_L$ by $Tr_\omega(\rho(L))$ we get the canonical trace $\tau$ for any ultrafilter $\omega$. The multiplicity $m(\lambda)$ of an eigenvalue $\lambda\ge 0$ is the cardinality of the level set $\{g\in G:\ell(g)=\lambda\}$ while its cumulated multiplicity $M(\lambda)$ is the cardinality of the sub-level set $\{g\in G:\ell(g)\le\lambda\}$. In case $m(\lambda)=o(M(\lambda))$ as $\lambda\to+\infty$, the eigenvalue counting function is asymptotically continuous, $\rho(L)$ is a density and the volume form $\Omega_L$ coincides with the trace state $\tau$ for any ultrafilter $\omega$. This is the case of the word length function $\ell$ of a system of generators for a discrete group $G$ with sub-exponential growth \cite{deH}.

\begin{ex} We apply the previous argument to the case where $G=\F_p$ is the free group with $p$ generators and exponential growth and $L$ is the multiplication operator by the length function $\ell$. Extending the argument of subsection \ref{mult}, we get easily $Tr(a\rho(L)^s)=a_e Tr(\rho(L)^s)=\tau(a)Tr(\rho(L)^s)$ for $a\in C^*_{red}(G)$ and $s\in \C$, $Re(s)>1$. Proposition \ref{GL} allows to reach the volume form as a residue :
$$ \Omega_L(a)=\frac{2p-2}{(2p-1)\rm{Log}(2p-1)}\,\tau(a)=\lim_{s\in \C\,,\;\frak{Re}(s)>1\,,\;s\to 1}(s-1)\, Tr(a\,\rho(L)^s)\,,\;a\in C^*_{red}(G)\,.  $$
\end{ex}

\section{Further examples}

\subsection{The Toeplitz C$^*$-algebra I}\label{toeplitz1}
Let $A\subset\B(l^2(\mathbb{N}))$ be the Toeplitz C$^*$-algebra generated by the shift operator $S$ on $l^2(\mathbb{N})$:
\[
S e_n=e_{n+1}\qquad n\in\N
\]
and let $L$ be the multiplication operator on $l^2(\mathbb{N})$ given by
\[
(Lu)(n):=nu(n)\qquad n\in \mathbb{N},\quad u\in l^2(\mathbb{N})\, .
\]
Its spectrum $sp(L)=\N$ is discrete with all multiplicities equal to one and counting function $N_L(x)=[x]+1$ for $x\ge 0$. Hence $N_L(L)=L+1$ and $\rho(L)=(L+1)^{-1}$. Since $N_L(x)\sim \varphi(x):=x$ as $x\to+\infty$, $N_L$ is asymptotically continuous and measurability holds true. In this case $\zeta_L$ coincides with the Riemann $\zeta$-function $\zeta_0$ and since the remainder function $N_L(x)-\varphi(x)=1-(x-[x])$ is bounded, applying Theorem \ref{merocor} for all $\alpha\in (0,1)$, we obtain the known fact that $\zeta_0$ is meromorphic in the open right half plane with a simple pole at $s=1$ and unital residue. The volume forms
\[
\Omega_L(a)=Tr_\omega (a\varphi(L)^{-1})=Tr_\omega (aL^{-1})\qquad a\in A
\]
are states on $A$ vanishing on the ideal $\K(l^2(\mathbb{N}))$ of compact operators. Since $A$ is an extension in the sense of [D]
\[
  \begin{tikzcd}
    0 \arrow{r} & \mathcal{K}(l^2(\N)) \arrow{r} & A \arrow{r}{\sigma} & C(\mathbb{T}) \arrow{r} & 0,
  \end{tikzcd}
\]
where $\mathbb{T}$ is the unit circle and $\sigma(S)(z)=z$ for $z\in\mathbb{T}$, it follows that the states $\Omega_L$ are determined by probability measures $m_\omega$ on $\mathbb{T}$, $\Omega_L(a)=\int_\mathbb{T}\sigma(a)dm_\omega$ for all $a\in A$, and that they are thus traces on $A$ since $C(\mathbb{T})$ is commutative. Since $Tr(S^kL^{-s})=\sum_{n\ge 0}(e_n|S^kn^{-s}e_n)=\sum_{n\ge 0}n^{-s}(e_n|e_{n+k})=\delta_{k,0}\cdot \zeta_0(s)$ for all $s>1$ and $k\ge 0$, by formula \ref{CPS38} we have $\Omega_L(S^k)=\delta_{k,0}$. Hence all measures $m_\omega$ coincides with the Haar probability measure $m_H$ for any ultrafilter $\omega$ and $aL^{-1}$ is measurable for any $a\in A$.
%\[
%  \begin{tikzcd}
%    A \arrow{r}{f} \arrow[swap]{dr}{g\circ f} & B \arrow{d}{g} \\
%     & C
%  \end{tikzcd}
%\]

%\begin{rem}
%\vskip0.2truecm
%1. \\
%Si potrebbe cercare di estendere la formula del residuo di Wocinscki-Adler-Guillemin per operatori pseudo-differenziali di ordine $\le 0$ a C$^*$-algebre $A$ che sono estensioni tramite l'algebra dei compatti $\mathcal{K}$ di un'algebra commutativa $C(X)$ su uno spazio di Hausdorff compatto. La composizione della mappa quoziente o mappa simbolo con una misura di probabilita' su $X$ fornisce una traccia su $A$. Vedi libro di R. Douglas e lavoro di Choi-Tsui sulla classificazione di mappe traccia tra C$^*$-algebre.
%\vskip0.2truecm

\subsection{The density of Euclidean domains having infinite volume}
In Euclidean domains with infinite volume $\Omega$, the Weyl's asymptotic cannot hold true for the Laplacian $L$ with Dirichlet boundary conditions on $\partial\Omega$, even if the spectrum is discrete. B. Simon determined in \cite{S} (Theorem 1.5) the asymptotic behavior of $N_L$ for certain planar domains of infinite volume. For example, when
$\Omega:=\{(x,y)\in \R^2:|xy|\le 1\}$ one has
\[
{\hat\mu_L}(t)=Z_L(t)\sim \frac{1}{\pi}\cdot t^{-1}\rm{Log}(t^{-1})\qquad t\to 0^+
\]
%so that $k_n=\lim_{t\to 0^+}\frac{{\hat\mu_L}(nt)}{{\hat\mu_L}(t)}=\frac{1}{n}$ for $n\ge 1$ and one verifies that $\mu_0(dx)=dx$.
from which, by Proposition \ref{karamata}, one derives the asymptotic behaviour
\[
N_L(x)\sim \frac{1}{\pi}\cdot x\cdot{\rm Log} (x)\qquad x\to+\infty.
\]
Hence $N_L$ is asymptotically continuous and, by Proposition \ref{continuity}, $\rho(L)$ is a density. The volume states read
\[
\Omega_L(T)=\pi\cdot Tr_\omega(TL^{-1}\rm{Log}^{-1}(L+I))\qquad T\in\B(L^2(\Omega,dx))
\]
and they determine probability measures $\nu_\omega$ on $\Omega$ by $\int_\Omega f\cdot d\nu_\omega:=\Omega_L(M_f)$ for $f\in C_0(\Omega)$.
%Let $d:\Omega\to [0,+\infty)$ be the distance function from the boundary
%\[
%d(x):=\inf\{|x-y|:y\in\partial\Omega\}.
%\]
%A domain $\Omega\subset\R^n$ with $n\ge 2$, is said {\it strongly regular} if $L\ge c/d^2$ in the quadratic form sense for some $c>0$ ($L$ being the Dirichlet Laplacian). E.B. Davies showed that the heat semigroup generated by $L$ is nuclear if and only if
%\[
%\int_\Omega e^{-\frac{\beta}{d^2}} dm<+\infty\qquad {\rm for\,\,all\,\,} \beta>0.
%\]
%(this is the case, for example, if $d\in L^p(\Omega,m)$ for some $p\ge 1$). In these situations one has
%\[
%(8\pi \beta)^{-n/2}\int_\Omega e^{-\frac{8\pi^2 n\beta}{d^2}} dm\le {\rm Tr\,}(e^{-\beta L})\le (2\pi \beta)^{-n/2}\int_\Omega e^{-\frac{c\beta}{2d^2}} dm\qquad \beta>0.
%\]
%Are these examples of measurable densities?

\subsection{Kigami Laplacians on P.C.F. fractals} Let $K$ be a post critically finite, self-similar fractal set and $(\E,\mathcal{F})$ be the Dirichlet form associated to a fixed regular harmonic structure (with energy weights $0<r_1,\cdots,r_m<1$) in the J. Kigami's sense \cite{KL}. This quadratic form is closable with respect to any Bernoulli measure $m$ on $K$ (with weights $0<\mu_1,\cdots,\mu_m<1$ such that $\sum_{i=1}^m\mu_i=1$)  and we denote by $L$ the densely defined, nonnegative, self-adjoint operator on $L^2(K,m)$ associated to its closure. Set $\gamma_i:=\sqrt{r_i\mu_i}$ and define the {\it spectral dimension} $d_S$ as the unique positive number such that
\[
\sum_{i=1}^m \gamma_i^{d_S}=1.
\]
In the {\it non-arithmetic} case, where $\sum_{i=1}^m\mathbb{Z}\rm{Log}\gamma_i$ is a dense additive subgroup of $\R$, the asymptotic of the counting function $N_L$ follow a power law similar to the H. Weyl 's one for compact Riemannian manifolds (see \cite{KL} Theorem 2.4)
\[
N_L(x)\sim c\cdot x^{d_S/2},\qquad x\to+\infty
\]
where $c:=\Bigl[-\Bigl(\sum_{i=1}^m\gamma_i^{d_S}\rm{Log}\gamma_i\Bigr)^{-1}\cdot\int_\R e^{-d_St}R(e^{2t})dt\Bigr]$ and $R(x):=N_L(x)-\sum_{i=1}^mN_L(r_i\mu_i x)$. Consequently, the spectral weight $\rho(L)$ is measurable and it is in fact a density. Evaluating the volume forms $\Omega_L$ on the multiplication operators $M_g\in \mathcal{B}(L^2(K,m))$ by continuous functions $g\in C(K)$, one gets positive states on $C(K)$ represented by probability measures $\nu_\omega$ on $K$
\[
\Omega_L(M_g)=c^{-1}\cdot Tr_\omega(M_g\cdot L^{-d_S/2})=\int_K g\cdot d\nu_\omega.
\]
%For measurable spectral triples on the Sierpinski gasket $K$ see [...], [...].
%\vskip0.2truecm\noindent
%In {\it arithmetic} case, where $\sum_{i=1}^m\mathbb{Z}\rm{Log}\gamma_i$ is a discrete additive subgroup of $\R$ with period $T$,  $N_{\Delta_\mu}$ has instead the behaviour
%\[
%N_{\Delta_\mu}(x)\sim G\bigl(\rm{Log} \sqrt{x}\bigr)\cdot x^{d_S/2},\qquad x\to+\infty
%\]
%where $G:\R\to\R$ is a cadlag, $T$-periodic function satisfying
%\[
%0<\liminf_{x\to+\infty} G(x)<\liminf_{x\to+\infty} G(x)<+\infty
%\]
%that can be represented as
%\[
%G(t)=-T\Bigl(\sum_{i=1}^m\gamma_i^{d_S}\rm{Log}\gamma_i\Bigr)^{-1}\cdot\sum_{j\in\Z}e^{-d_S(t+jT)}R(e^{2(t+jT)}).
%\]
%The non constant character of $G$ is influenced by the high self-similarities of the fractal $K$, the Bernoulli measure $\mu$ and the Dirichlet form $\E$ which causes large multiplicity jumps. In particular, for the standard Dirichlet form and the symmetric Bernoulli measure on the Sierpinski gasket in $\R^m$, $m\ge 2$ ($r_i=3/5$, $\mu_i=1/3$ for all $i=1,\cdots,m$) [Fukushima-Shima Thm 5.2] proved that $d_S=\frac{2\rm{Log} (m+1)}{\rm{Log} (m+3)}$ and
%\[
%c:=\limsup_{k\to+\infty}\frac{M_{k+1}}{M_k}...........................
%\]
%{\bf Check carefully Fukushima-Shima to deduce that this is a non measurable case and compare with the paper of Kigami-Lapidus}.
\section{Volume traces from densities}

We denote by $\A_L$ the so called {\it Lipschitz algebra} of $L$, i.e.
\[
\A_L:=\{a\in \B(h)\,,\; [a,L] \text{ is bounded }\}.
\]

\subsection{Statement of the results.}

The purpose of the whole section is to prove the following theorem, together with two main corollaries, providing conditions which ensure the existence of {\it hypertraces or amenable traces} (see \cite{B}, \cite{Co2}). In case of a finitely-summable spectral triple (\cite{Co1} page 68), the result, known as Connes' Trace Theorem is proved in \cite{Co2} Theorem 8 and Remark 10 b). In case of a $(d,\infty)$-summable spectral triple, a proof of the the result, stated in \cite{Co3} Chapter IV.2 Proposition 15, is provided in \cite{CGS}.
\begin{thm}\label{residue} (Trace Theorem) Suppose that the counting function satisfies
\[
N_L(x)\sim \varphi(x)\qquad x\to+\infty
\]
for some continuous, increasing function $\varphi:\R_+\to\R_+$ such that
\begin{equation}
\varphi'\in L^\infty_{\rm loc}(\R_+),\qquad \label{subexp} {\rm ess-}\limsup_{x\to +\infty} \varphi'(x)/\varphi(x)=0.
\end{equation}
Then the following limit properties hold true
\vskip0.2truecm\noindent
(1.a) For $s>1$, $\varphi(L)^{-s}$ is trace-class and $ \lim_{s\downarrow 1} (s-1)\, Tr\Big(\,\varphi(L)^{-s}\,\Big)=1$\,.

(1.b) For $a \in \A_L$, one has
\begin{equation}\label{Trconvergence}
\lim_{s\downarrow 1} \big(s-1)\, Tr\Big(\,\big|\,\big[ a\,,\,\varphi(L)^{-s}\big]\,\big|\,\Big) =0
\end{equation}

(1.c) For $a\in \A_L$ and $b\in \B(h)$, one has
\begin{equation}\label{stoone}
\lim_{s\to 1} \big(s-1)\, Tr\Big( (ab-ba)\varphi(L)^{-s}\Big) =0\,.
\end{equation}

\medskip (2) Hypertrace properties.
\vskip0.2truecm
For $s>1$ define the linear functionals
\[
\omega_s(b):=(s-1)Tr(b\,\varphi(L)^{-s})\qquad b\in\B(h).
\]
As $s\downarrow 1$, the limit point set of $\{\omega_s\in\B(h)^*:s>1\}$ is not empty and any such limit linear form $\tau$ is a state on $\B(h)$ with the following properties :
\vskip0.2truecm\noindent
 (2.a) $\tau$ vanishes on the algebra $\K(h)$ of compact operators\,;
 \vskip0.2truecm\noindent
 (2.b) $\tau$ is a hypertrace on the uniforme closure $A$ of the Lipschitz algebra $\A_L$\,:
 \begin{equation}
 \tau(ba)=\tau(ab)\,,\; a\in A\,,\, b\in \B(h)\,;
 \end{equation}
\vskip0.2truecm\noindent
 (2.c) The restriction of $\tau$ to $A$ is a tracial state.

 \medskip 3. Any volume form $\Omega_L(a)=Tr_\omega(a\,\rho(L))$ on $\B(h)$ is an hypertrace and a tracial state on $A$ (with $ \omega$ as in subsection \ref{DixmierTraces}).
 \end{thm}

 The first corollary is just a variation on the conclusions, with the same assumptions :
 \begin{cor} \label{residueF} Suppose that the counting function satisfies
\[
N_L(x)\sim \varphi(x)\qquad x\to+\infty
\]
for some continuous, increasing function $\varphi:\R_+\to\R_+$ such that
\begin{equation}
\varphi'\in L^\infty_{\rm loc}(\R_+),\qquad \label{subexp} {\rm ess-}\limsup_{x\to +\infty} \varphi'(x)/\varphi(x)=0.
\end{equation}
Then the following limit properties hold true
\vskip0.2truecm\noindent
(1.a) For $s>1$, $\rho(L)^{s}$ is trace-class and $ \lim_{s\downarrow 1} (s-1)\, Tr\Big(\,\rho(L)^{s}\,\Big)=1$\,.

(1.b) For $a \in \A_L$, one has
\begin{equation}\label{Trconvergence}
\lim_{s\downarrow 1} \big(s-1)\, Tr\Big(\,\big|\,\big[ a\,,\,\rho(L)^{s}\big]\,\big|\,\Big) =0
\end{equation}

(1.c) For $a\in \A_L$ and $b\in \B(h)$, one has
\begin{equation}\label{stoone}
\lim_{s\to 1} \big(s-1)\, Tr\Big( (ab-ba)\rho(L)^{s}\Big) =0\,.
\end{equation}

\medskip

For $s>1$ define the linear functionals
\[
\omega_s(b):=(s-1)Tr(b\,\rho(L)^{s})\qquad b\in\B(h).
\]
As $s\downarrow 1$, the limit point set of $\{\omega_s\in\B(h)^*:s>1\}$ is not empty and any such limit linear form $\tau$ is a state on $\B(h)$ with the following properties :
\vskip0.2truecm\noindent
 (2.a) $\tau$ vanishes on the algebra $\K(h)$ of compact operators\,;
 \vskip0.2truecm\noindent
 (2.b) $\tau$ is a hypertrace on the uniforme closure $A$ of the Lipschitz algebra $\A_L$\,:
 \begin{equation}
 \tau(ba)=\tau(ab)\,,\; a\in A\,,\, b\in \B(h)\,;
 \end{equation}
\vskip0.2truecm\noindent
 (2.c) The restriction of $\tau$ to $A$ is a tracial state.

 \medskip 3. Any volume form $\Omega_L(a)=Tr_\omega(a\,\rho(L))$ on $\B(h)$ is an hypertrace and a tracial state on $A$ (with $ \omega$ as in subsection \ref{DixmierTraces}).
 \end{cor}

The second corollary provides a sufficient condition for the assumptions above to hold true. It requires that the relative multiplicities vanish faster than the spectral gaps of $L$:

 \begin{cor}\label{sufficientF}
Suppose that the following conditions on the spectrum of $L$ are satisfied :
\[
\lim_{k\to \infty} m_{k}/M_k=0\quad \text{ and }\quad\frac{m_k}{M_k}=
%\frac{M_{k+1}}{M_k}-1=
o(\tl_{k+1}(L)-\tl_k(L))\,\,{\it as}\,\,x\to+\infty\,.
\]
Then the assumptions of Corollary \ref{residueF} are satisfied, so that all of its conclusions hold true. In particular, they hold true if
%$\ds \lim_{n\to \infty} \frac{M_{k+1}}{M_k}=1$ (i.e.
$N_L$ is asymptotically continuous and the spectral gaps are uniformly bounded away from zero
\[
\liminf_{k\to \infty} (\tl_{k+1}(L)-\tl_k(L))>0.
\]
\end{cor}
Passing from the operator $L$ to a monotone functional calculus $f(L)$ of it, multiplicities remain unchanged but gaps $f(\tl_{k+1}(L))-f(\tl_k(L))$ may vary. However, conditions involving gaps in the Corollary above are still satisfied, for example, if $f\in C^1(\R_+)$ and $\inf f'>0$.

%\subsection{The Toeplitz algebra II}
\begin{ex}({\bf The Toeplitz C$^*$-algebra II})
Let $\mathbb{P}=\{2,3,5,\cdots\}$ be the set of prime numbers and let $S':l^2(\mathbb{P})\to l^2(\mathbb{P})$ be the shift operator defined as $(S'u)(2)=0$ and
$(S'u)(p)=u(p')$ where $p'\in\mathbb{P}$ denotes the greatest prime strictly less than $p\in\mathbb{P}$ for $p\ge 3$. $S'$ is an isometry $S'^*S'=I$ which generates the Toeplitz C$^*$-algebra $A'\subset\B(l^2(\mathbb{P}))$.\\
Let $J$ be the operator on $l^2(\mathbb{P})$ defined by
\[
(Ju)(p):=pu(p)\qquad p\in \mathbb{P},\quad u\in l^2(\mathbb{P})\, .
\]
We have ${\rm Sp}(J)=\mathbb{P}$, all multiplicities equal to one and, by the Prime Number Theorem, $N_J(x)\sim \varphi(x):=x/{\rm Log} (x)$ for $x\to +\infty$ so that $N_J$ is asymptotically continuous and $\rho(J)$ is a density. Since $\varphi'(x)=({\rm Log}( x))^{-1}-({\rm Log} (x))^{-2}> 0$ for $x>3$, $\varphi$ is strictly increasing and since $\varphi'(x)/\varphi(x)=x^{-1}-(x{\rm Log} (x))^{-1}\to 0$ as $x\to+\infty$, by Theorem \ref{residue} we have that the volume state on $\B(l^2(\mathbb{P}))$
\[
\Omega_J(a)=Tr_\omega(a\varphi(J)^{-1})=Tr_\omega(a\rm{Log} J/J)\qquad a\in A'
\]
is an hypertrace on the Toeplitz C$^*$-algebra, vanishing on the ideal $\K(l^2(\mathbb{P}))$. Since for $s>1$\\
$Tr(S'^k\varphi(J)^{-s})=\sum_{p\in\mathbb{P}}(\delta_p |S'^k\varphi(J)^{-s}\delta_p)=\sum_{p\in\mathbb{P}}(p/{\rm Log} (p))^{-s}(\delta_p|S'^k\delta_p)=\delta_{k,0}\cdot \sum_{p\in\mathbb{P}}(p/{\rm Log} (p))^{-s}\\
\sim\delta_{k,0}\cdot(s-1)^{-1}$ as $s\to 1^+$, by formula \ref{CPS38} we have $\Omega_J(S'^k)=\delta_{k,0}$ for any Dixmier ultrafilter $\omega$. Analogously to the situation of section \ref{toeplitz1}, one has $A'/\K(\mathbb{P})\simeq C(\mathbb{T})$ and the induced measure on the circle $\mathbb{T}$ is again the Haar probability measure.
%Since $\varphi(x)=x/\rm{Log} x$ is such that
%$\varphi'(x)=(\rm{Log} x-1)/(\rm{Log} x)^2$ and $\varphi'(x)/\varphi(x)=1/x-1/x\rm{Log} x\to 0$ as $x\to+\infty$, it follows that the residue %state $\phi_L$ is a state on $A$.
\vskip0.2truecm\noindent
To compare the situation the described in \ref{toeplitz1} to the present one, let us notice first that, since
\[
[L,S]e_n=LSe_n-SLe_n=Le_{n+1}-nSe_n=(n+1)e_{n+1}-ne_{n+1}=e_{n+1}=Se_n\qquad n\in\N,
\]
we have $[L,S]=S$ so that the $^*$-algebra $\mathcal{A}_L$ generated by $S$ contains all commutators $[L,a]$ for any $a\in \mathcal{A}_L$. On the other hand, the commutator $[J,S']$ is unbounded since
\[
([J,S']u)(p)=p(S'u)(p)-S'(Ju)(p)=pu(p')-(Ju)(p')=(p-p')u(p')\qquad 3\le p\in\mathbb{P}
\]
and it is known that the {\it prime gap} $g(p'):=p-p'$ can be arbitrarly large. Moreover, $[\rm{Log} J,S']$ is compact. In fact for $3\le p\in\mathbb{P}$
\[
\begin{split}
([\rm{Log} J,S']u)(p)&=(\rm{Log} p)(S'u)(p)-S((\rm{Log} J)u)(p)=pu(p')-((\rm{Log} J) u)(p')=(\rm{Log} (p/p'))u(p')\\
&=(\rm{Log} (p/p'))(S'u)(p)
\end{split}
\]
and it is known that $\lim_{p\to +\infty} p/p'=\lim_{p'\to +\infty} (1-g(p')/p')=1$. However, A.E. Ingham \cite{Ing} showed that there exists $\alpha\in (0,3/8)$ such that $p-p'\le p'^{1-\alpha}$ for sufficiently large $p$. Hence, for this fixed value of $\alpha$ and for sufficiently large $p$ we have
\[
0\le p^\alpha-p'^\alpha=p'^\alpha\Bigl[\Bigl(1+\frac{p-p'}{p'}\Bigr)^\alpha-1\Bigr]\le p'^\alpha\alpha\frac{p-p'}{p'}=\alpha\frac{p-p'}{p'^{1-\alpha}}\le \alpha
\]
so that, for some constant $C\ge \alpha $, we have $0\le p^\alpha-p'^\alpha\le C$ for all $p\in\mathbb{P}$.
Then $([J^\alpha,S']u)(2)=J^\alpha(S'u)(2)-S'(J^\alpha u)(2)=0$ and for $3\le p\in\mathbb{P}$
\[
\begin{split}
([J^\alpha,S']u)(p)&=p^\alpha(S'u)(p)-S'(Ju)(p)=p^\alpha u(p')-(Ju)(p')=(p^\alpha -p'^\alpha)a(p')\\
&= (p^\alpha -p'^\alpha)\cdot (S'u)(p).
\end{split}
\]
It follows that $\|[J^\alpha,S']\|\le C$ and that all commutators $[J^\alpha,a]$ are bounded for any $a$ in the $^*$-subalgebra $\A_{J^\alpha}'\subset A'$ generated by $S'$. Notice that $\Omega_J=\Omega_{J^\alpha}$ since $J^\alpha$ is an increasing, unbounded function of $J$. In conclusion, even if the C$^*$-algebras $A$ and $A'$ are isomorphic and their hypertraces correspond $\Omega_L\simeq \Omega_J$, these structures differ from a metric point of view since their {\it Lipschitz} algebras are not isomorphic $\A_L\nsim \A_{J^\alpha}'$.
\end{ex}

\begin{ex}
The Dirac operator $D$ of a spectral triple $(\A,h,D)$ defined on a C$^*$-algebra $A$, represented in a Hilbert space $h$, and associated to a filtration of $h$ as in subsection 2.3, has spectrum $\N$. All spectral gaps are equal to $1$ so that the second condition in Corollary 6.3 is satisfied. As soon as the growth of the filtration satisfies $\lim_{k\to +\infty}M_{k+1}/M_k=1$, the spectral weight $\rho(D)$ is then a density and the volume states $\Omega_D$ are hypertraces for any Dixmier ultrafilter $\omega$.
\end{ex}
%Bost and Connes showed that the operator $L$ is the Bose-Einstein second quantization of $J$.

%\begin{ex}
%If $Z_{D^2}(t)\sim c\cdot t^{-d/2}$ then $N_{D^2}([0,x))\sim c\cdot x^{d/2}$ so that $N_{D}([0,x))\sim c\cdot x^{d}$ and $|D|^{-d}\in\mathcal{L}^{(1,\infty)}(h)$ and viceversa.
%Hence any $(d,\infty)$-summable k-cycle is measurable.
%\end{ex}
\subsection{Relationship with subexponential growth.}

Here, we assume that the assumptions of Theorem \ref{residue} hold true.  As first consequence, $L$ has subexponential spectral growth rate.
\begin{lem} For any $\beta>0$, the partition function is finite $Z_L(\beta):=Tr(e^{-\beta L})<+\infty$.
\end{lem}
\begin{proof}
Condition (\ref{subexp}) on $\varphi$ implies that, for any fixed $\beta>0$, the nonnegative function $x\to e^{-\beta x}\varphi(x)$ has a derivative $(\varphi'(x)-\beta\varphi(x))e^{-\beta x}$ which is eventually negative so that it admits a limit at $+\infty$. Hence, for all fixed
$\beta>0$, $\lim_{x\to +\infty} e^{-2\beta x}\varphi(x)=0$. In particular, $\lim_{n\to \infty} e^{-2\beta n}\varphi(n)\to 0$, so that $\limsup_{n\to \infty} \sqrt[n]{\varphi(n)}\leq e^{2\beta}$. Since this holds for any $\beta>0$, we get $\limsup_{n\to \infty}\sqrt[n]{\varphi(n)}\leq 1$ (and indeed $\lim_{n\to \infty}\sqrt[n]{\varphi(n)}=1$). Appling \cite{CS} Lemma 3.13 we get the result.
\end{proof}

\subsection{Preparatory results.} In this section we assume the hypotheses of Theorem \ref{residue}.

\smallskip
\begin{lem}\label{res0} For $s>1$ \,, $\varphi(L)^{-s}$ is  trace class, with
$\lim_{s\to 1+} (s-1) Tr(\varphi(L)^{-s})=1$.
\end{lem}
\begin{proof}
The assumptions made imply that $\varphi$ is a continuous function, which in turns implies
$\lim_k M_k/M_{k-1}=1$ (by Proposition \ref{continuity}). Then, by Proposition \ref{asymptotics}, we get
$N_L(\lambda_n(L))\sim n$ as $n\to+\infty$ (eigenvalues numbered with repetition according to the multiplicity) and thus
$\la_n(\varphi(L)^{-1})\sim 1/n$. The result follows easily.
\end{proof}

The assumptions on $\varphi$ lead to the following technical result.
\begin{lem}\label{quotients} For $s>1$ and $N\in \N^*$, one has
\begin{equation} \sup\nolimits_{k>\ell\geq N} \frac{\varphi(\tl_k(L))^s-\varphi(\tl_\ell(L))^s}{\varphi(\tl_k(L))^s(\tl_k(L)-\tl_\ell(L))} \leq  \text{ess-}\sup\nolimits_{x\geq \tl_N(L)} s\, \frac{\varphi'(x)}{\varphi(x)}\,.
\end{equation}
\end{lem}
\begin{proof}
To short notation set $\la_k:=\la_k(L)$, etc.. Keeping in mind the fact that $\varphi$ is increasing, we write for $k>\ell\geq N$
\begin{equation*}\begin{split} \frac{\varphi(\tl_k)^s-\varphi(\tl_\ell)^s}{\varphi(\tl_k)^s} &= \int_{\tl_\ell}^{\tl_k} \frac{s\varphi(x)^{s-1}\varphi'(x)}{\varphi(\tl_k)^s}dx \\
&\leq  \int_{\tl_\ell}^{\tl_k} \frac{s\varphi(x)^{s-1}\varphi'(x)}{\varphi(x)^s}dx \\
&\leq (\tl_k-\tl_\ell)\, \text{ess-}\,\sup\nolimits_{\tl_\ell\leq x \leq \tl_k} s\,\varphi'(x)/\varphi(x) \\
&\leq (\tl_k-\tl_\ell)\, \text{ess-}\,\sup\nolimits_{\tl_N\leq x } s\,\varphi'(x)/\varphi(x) \,.
\end{split}\end{equation*}
\end{proof}

\smallskip
Let $L^2(h)$ be the space of Hilbert-Schmidt operators on $h$ with norm $||\Phi||_{2}=Tr(\Phi^*\Phi)^{1/2}$ and corresponding scalar product $\langle\Psi,\Phi\rangle_2=Tr(\Psi^*\Phi)$. For $\ell\ge 1$, let us denote by $\pi_\ell$ the orthogonal projection onto the eigenspace corresponding to the eigenvalue $\lambda_\ell (L)$.
%\vskip0.2truecm\noindent
\begin{lem}\label{stime}
Let $\{\alpha_{k,\ell}\}_{k,\ell \geq 1}\subset\C$ be a bounded set and $T$ a bounded operator on $h$. Then
\[
\sum_{k,\ell} \alpha_{k,l} \,\pi_k T \pi_\ell\, \varphi(L)^{-s/2},\qquad s>1,
\]
is a Hilbert-Schmidt operator and the following estimate holds true :
\begin{equation}\label{HS}
\left\| \sum_{k,\ell} \alpha_{k,l} \pi_k T \pi_\ell \,\varphi(L)^{-s/2}\right\|_2 \leq \Bigl(\sup_{k,\ell} |\alpha_{k,\ell}|\Bigr)\cdot ||T|| \cdot Tr\big(\varphi(L)^{-s}\big)^{1/2}\,.
\end{equation}
\end{lem}
\begin{proof} As the right and left actions of $\B(h)$ on $L^2(h)$ commute, for each $k,\ell\geq 1$ we define an orthogonal projection $p_{k,\ell}$ in $\B(L^2(h))$ by
\[
p_{k,\ell}\Phi = \pi_k\Phi\pi_\ell\,,\; \Phi \in L^2(h)\,.
\]
We have obviously $\sum_{k,\ell} p_{k,\ell}=I$, so that the operator norm of $\sum_{k,\ell} \alpha_{k,\ell} p_{k,\ell}$ acting on $L^2(h)$ is $\sup_{k,\ell} |\alpha_{k,\ell}|$. We get the result writing\\
$\sum_{k,\ell} \alpha_{k,l}\, \pi_k T \pi_\ell \varphi(L)^{-s/2}=\sum_{k,\ell} \alpha_{k,l}\, \pi_k T  \varphi(L)^{-s/2} \pi_\ell = \big(\sum_{k,\ell} \alpha_{k,\ell}\, p_{k,\ell}\big) T \varphi(L)^{-s/2}$.
\end{proof}

\begin{prop}\label{mezzaconvergenza}
For any $s>1$ and $N\geq 1$, set
\[
Y_N(s)=\sum_{k>\ell\geq N} \frac{\varphi(\tl_k(L))^s-\varphi(\tl_\ell(L))^s}{\varphi(\tl_k(L))^s\,(\tl_k(L)-\tl_\ell(L))}\pi_k\, [L,a]\, \pi_\ell\, \varphi(L)^{-s/2}\,.
\]
i) One has $Y_N(s)\in L^2(h)$ with $||Y_N(s)||_2\leq  s\big( \sup_{x\geq \tl_N} \varphi'(x)/\varphi(x)\big) \; ||\,[L,a]\,|| \;
Tr(\varphi^{-s}(L))^{1/2}$\,.

ii) $\lim_{s\downarrow 1} (s-1)^{1/2}\,||Y_1(s)||_2=0$\,.
\end{prop}
\begin{proof}
To short notation set $\la_k:=\la_k(L)$, etc.. For i), apply Lemmas \ref{quotients} and \ref{stime}. ii) Fix $\e>0$ and $N\ge 1$ such that
\[
{\rm ess-}\sup_{x\geq \tl_N} \varphi'(x)/\varphi(x)\leq \e.
\]
On one hand, $\displaystyle Y_1(s)-Y_N(s)=\sum_{1\leq \ell\leq N\,,\,k\geq \ell} \frac{\varphi(\tl_k)^s-\varphi(\tl_\ell)^s}{\varphi(\tl_k)^s\,(\tl_k-\tl_\ell)}\pi_k\, [L,a]\, \pi_\ell\, \varphi(L)^{-s/2}$ has a Hilbert-Schmidt norm less than
$C\,\left\| \big(\sum_{\ell=1}^N \pi_\ell\,\big)\, \varphi(L)^{-s/2}\right\|_2$ for some constant $C$ depending only on $\varphi$ and $||\,[L,a]\,||$, by Lemma \ref{stime}. Hence $||Y_1(s)-Y_N(s)||_2$  is bounded independently of $s$. For $s$ close enough to $1$, we  then have $(s-1)^{1/2}\,||Y_1(s)-Y_N(s)||_2\leq \e$.

\smallskip On the other hand, applying i), we get $||Y_N(s)||_2\leq \e\,s\,||\,[L,a]\,|| \, Tr(\varphi^{-s})^{1/2}$\,. Applying Lemma \ref{res0}, we get that, for $s$ close enough to $1$, $(s-1)^{1/2}Y_N(s)$ has a Hilbert-Schmidt norm less than $\e\cdot s\cdot \|[L,a]\| \cdot(1+\e)$\,.

\smallskip Summing up, we get that,  for $s$ close enough to 1, $(s-1)^{1/2}Y_1(s)$ has Hilbert-Schmidt norm less that $\e\cdot\big(s\,||\,[L,a]\,|| \,(1+\e)+1\big)$\,.
\end{proof}

\begin{prop}\label{mezzaconvergenza2}
Setting $\displaystyle Z_1(s)= \sum_{1\leq k \leq \ell} \varphi(L)^{-s/2}  \frac{\varphi(\tl_k(L))^s-\varphi(\tl_\ell(L))^s}{\varphi(\tl_\ell(L))^s\,(\tl_k(L)-\tl_\ell(L))}\pi_k\, [L,a]\, \pi_\ell\,$,\\
we have :

i) $Z_1(s)\in L^2(h)$ whenever $s>1$

ii) $\lim_{s\downarrow 1} (s-1)^{1/2} ||Z_1(s)||_2=0$.
\end{prop}
{\it Proof.} Up to a sign, $Z_1(s)^*$ is given by the same formula as $Y_1(s)$, with $\bar s$ substituted to $s$ and $a^*$ substituted to $a$. Apply Proposition \ref{mezzaconvergenza} i) and ii).
~\hfill $\square$

\begin{lem}\label{commutator} (Chain rule) For any $f\in C(\R)$ and $a\in \A_L$  we have, for $k,\ell\geq 1$, $k\not=\ell$\,:
\par \hskip1cm i) $\pi_k\,[a,f(L)]\,\pi_k=0$ and $\pi_k\,[a,f(L)]\,\pi_\ell = (f(\tl_k)-f(\tl_\ell))\,\pi_k\,a\,\pi_\ell$\,.
\par \hskip1cm ii) $\pi_k\,[L,a]\,\pi_\ell = (\tl_k-\tl_\ell)\,\pi_ka\pi_\ell\,.$
\par \hskip1cm iii) $\ds \pi_k\,[a,f(L)]\,\pi_\ell =\frac{ (f(\tl_k(L))-f(\tl_\ell(L)))}{\tl_k(L)-\tl_\ell(L)}\,\pi_k\,[L,a]\,\pi_\ell$\,.
\par In other words, by easily understood abuse of notation, we can write
\begin{equation}\label{comm}
[a,f(L)]= \sum_{k\not=\ell} \frac{ (f(\tl_k(L))-f(\tl_\ell(L)))}{\tl_k(L)-\tl_\ell(L)}\,\pi_k\,[L,a]\,\pi_\ell\,.
\end{equation}
\end{lem}
{\it Proof.} i) and ii) are straightforward. iii) is an obvious combination of i) and ii).
\hfill $\square$

\begin{prop}\label{Trconv} For any $a\in \A_L$ ( i.e. $[a,L]$ is bounded) one has
\begin{equation}
\lim_{s\downarrow 1} Tr\left(\,\big| \, [a,\varphi(L)^{-s}\,\big|\,\right) =0\,.
\end{equation}
\end{prop}
\begin{proof}
To short notation set $\la_k:=\la_k(L)$, etc.. Lemma \ref{commutator} allows to write
\begin{equation} \begin{split}
[a,\varphi(L)^{-s}]&=\sum_{k\not=\ell} \frac{ (\varphi(\tl_k)^{-s}-\varphi(\tl_\ell)^{-s})}{\tl_k-\tl_\ell}\,\pi_k\,[L,a]\,\pi_\ell \\
&=\sum_{k\not=\ell} \frac{ (\varphi(\tl_k)^{s}-\varphi(\tl_\ell)^{s})}{\varphi(\tl_k)^{-s}(\tl_k-\tl_\ell)\varphi(\tl_\ell)^{-s}}\,\pi_k\,[L,a]\,\pi_\ell \\
&= X^+(s)+X^-(s)
\end{split}\end{equation}
with
\begin{equation}\begin{split}
X^+(s)&=\sum_{k>\ell\geq 1} \frac{ (\varphi(\tl_k)^{s}-\varphi(\tl_\ell)^{s})}{\varphi(\tl_k)^{-s}(\tl_k-\tl_\ell)\varphi(\tl_\ell)^{-s}}\,\pi_k\,[L,a]\,\pi_\ell \\
&=\sum_{k>\ell\geq 1} \frac{ (\varphi(\tl_k)^{s}-\varphi(\tl_\ell)^{s})}{\varphi(\tl_k)^{-s}(\tl_k-\tl_\ell)}\,\pi_k\,[L,a]\,\pi_\ell \,\varphi(L)^{-s} \\
&=Y_1(s)\varphi(L)^{-s/2}
\end{split}\end{equation}
while
\begin{equation}\begin{split}
X^-(s)&=\sum_{1\leq k < \ell} \frac{ (\varphi(\tl_k)^{s}-\varphi(\tl_\ell)^{s})}{\varphi(\tl_k)^{-s}(\tl_k-\tl_\ell)\varphi(\tl_\ell)^{-s}}\,\pi_k\,[L,a]\,\pi_\ell \\
&=\sum_{1\leq k<\ell}\varphi(L)^{-s} \frac{ (\varphi(\tl_k)^{s}-\varphi(\tl_\ell)^{s})}{\varphi(\tl_\ell)^{-s}(\tl_k-\tl_\ell)}\,\pi_k\,[L,a]\,\pi_\ell \, \\
&=\varphi(L)^{-s/2} Z_1(s)\,.
\end{split}\end{equation}
Let $X^+(s)=u^+(s)\,|\,X^+(s)\,|$ be the polar decomposition of $X_+(s)$. Applying Lemma \ref{res0} and Proposition
\ref{mezzaconvergenza} (2), we get
\begin{equation} \begin{split}
Tr(\,|X^+(s)|\,)&=Tr(u^+(s)^*X^+(s)) \\ &=Tr(u^+(s)^*Y_1(s)\varphi(L)^{-s/2}) \\
&=Tr(\varphi(L)^{-s/2}u^+(s)^* Y_1(s)> \\
&\leq ||u^+(s)\varphi(L)^{-\bar s/2}||_2\,||Y_1(s)||_2 \\
& = O((Re(s)-1)^{-1/2})o((Re(s)-1)^{-1/2})=o(Re(s)-1)^{-1}\,,
\end{split}\end{equation}
which proves $(s-1)Tr(\,|X^+(s)|\,)\to 0$ as $s\downarrow 1$.

\smallskip A similar argument, {\it mutatis mutandis}, provides $(s-1)Tr(\,|X^-(s)|\,)\to 0$ as $s\downarrow 1$.
\end{proof}

\begin{lem}\label{compact} If $T$ is a compact operator, then
$$\lim_{s\downarrow 1} (s-1)Tr(T\,\varphi(L)^{-s})=0\,.$$
\end{lem}
\begin{proof}
Fix $\e>0$ and $T_0$ a finite rank operator such that $\|T-T_0||\leq \e$. On one hand, $\lim_{s\downarrow 1} Tr\big(T_0\varphi(L)^{-s}\big)=Tr\big(T_0\varphi(L)^{-1}\big)$ exists, so that $\lim_{s\downarrow 1}(s-1)Tr\big(T_0\varphi(L)^{-s}\big)=0$,
which means $\left| Tr\big(T_0\varphi(L)^{-s}\big)\right|\leq \e$ for $s$ close to $1$. On the other hand, one has for $s>1$,
\[
(s-1)\left| Tr\big((T-T_0)\varphi(L)^{-s}\big)\right|\leq \e (s-1)Tr(\varphi(L)^{-s})
\]
with, by Lemma \ref{res0}, $(s-1)Tr(\varphi(L)^{-s})\leq 1+\e$ for $s>1$ close to $1$. Summing up, we have $(s-1)\left| Tr(T\varphi(L)^{-s})\right| \leq \e(2+\e)$ for $s>1$ close to $1$.
\end{proof}

\medskip \subsection{Proofs of the theorem and its corollaries}\-
\vskip0.2truecm
{\it Proof of Theorem \ref{residue}.} (1.a) is Lemma \ref{res0}. (1.b) is Proposition \ref{Trconv} and (1.c) is an obvious consequence of (1.b). In (2), the fact that the $\Omega_s$ are bounded as $s\to 1+$ and that a limit linear form is a state is a consequence of Lemma \ref{res0}. (2.a) comes from Lemma \ref{compact}. (2.b.) and (2.c) come from (1.b) and (1.c).
~\hfill $\square$

\medskip {\it Proof of Corollary \ref{residueF}} We have $\varphi^{-1}=N_L^{-1}(1+g)$ with $g\,: \R_+\to \R_+$ vanishing at infinity. This implies $\varphi(L)^{-1} = N_L(L)^{-1}(I+T)$ with $T$ a compact operator commuting with $L$, $N_L(L)$ and $\varphi(L)$. Apply Lemma \ref{compact} repeatedly for substituting $N_L(L)^{-1}$ to $\varphi(L)^{-1}$ in every successive item of Theorem \ref{residue}.

~\hfill $\square$

\medskip{\it Proof of Corollary \ref{sufficientF}.} Let $\varphi$ be the continuous piecewise affine function on $\R_+$ interpolating affinely between the points
$\tl_k(L)$ and $\tl_{k+1}(L)$,
%N_L(\tl_k)=M_k)$,
i.e.
$$\varphi(x)=M_k+(x-\tl_k(L))\frac{M_{k+1}-M_k}{\tl_{k+1}(L)-\tl_k(L)}\quad \text{ whenever }\quad x\in [\tl_k(L),\tl_{k+1}(L)]\,.$$
This is the function constructed in Proposition \ref{continuity} where it is shown to be asymptotically equivalent to $N_L$,
provided that $M_{k+1}/M_k$ tends to $1$ as $k\to \infty$.

$\varphi$ is differentiable on each interval $(\tl_k,\tl_{k+1})$ with derivative $\displaystyle \varphi'(x)=\frac{M_{k+1}-M_k}{\tl_{k+1}-\tl_k}$. Moreover, for $x\in (\tl_k,\tl_{k+1})$ we have $\varphi(x)\geq M_k$ and $\ds \frac{\varphi'(x)}{\varphi(x)}\leq\left(\frac{M_{k+1}}{M_k}-1\right)\,\frac{1}{\tl_{k+1}-\tl_k}$ and by hypothesis
%In this product, the first factor tends to $0$ as $x\to +\infty$ while the second one remains bounded. Hence
we have $\ds \lim_{x\to +\infty} \frac{\varphi'(x)}{\varphi(x)}=0$.~\hfill $\square$

%\begin{rem} In Corollary \ref{sufficientF}, we can weaken the assumptions : if $M_{k+1}/M_k\to 1$, i.e. $m_{k+1}/M_k\to0$, and $\ds \frac{m_{k+1}}{M_k}=o(\tl_{k+1}-\tl_k)$, then the Corollary holds as well.
%\end{rem}

\subsection{Densities on C$^*$-algebras extensions} We conclude with a remark concerning densities and their volume forms on C$^*$-algebras extensions $A\subset \B(h)$ in the sense of \cite{D}, \cite{HR}
\[
  \begin{tikzcd}
    0 \arrow{r} & \mathcal{K} \arrow{r} & A\arrow{r}{\sigma} & C(X) \arrow{r} & 0,
  \end{tikzcd}
 \]
where $\K$ is the elementary C$^*$-algebra represented in $h$ with finite multiplicity and $X$ is a compact metrizable space.
%By construction, volume forms well behave with C$^*$-algebras extensions.
This framework include the Toeplitz extension and the extension generated by scalar, zero order $\Psi$DO on compact manifolds.

\begin{prop}(Volume forms on extension)
Assume the counting function $N_L$ to be asymptotically continuous. Then,  for any fixed Dixmier ultrafilter $\omega$,
\vskip0.2truecm\noindent
i) the volume form
\[
\Omega_L: \mathcal{B}(h)\to \mathbb{C}\qquad \Omega_L(T):={\rm Tr}_\omega(T\rho(L))
\]
is a state vanishing on the ideal $\mathcal{K}(h)$ of compact operators and thus it factorizes through a state on the Calkin algebra $\mathcal{Q}(h)= \mathcal{B}(h)/\mathcal{K}(h)$;
\vskip0.2truecm\noindent
ii) the restriction of $\Omega_L$ to $A$ is a trace that factorizes through a probability measure $\mu_\omega$ on $X$
\[
\Omega_L (a)=\int_X (\sigma(a))(x)\, \mu_\omega(dx)\qquad a\in A.
\]
Under the assumptions of Theorem \ref{residue} or Corollary \ref{sufficientF}, we also have
\vskip0.2truecm\noindent
iii) $\Omega_L$ is an hypertrace (or amenable trace state) vanishing on the ideal $\K$;
%\vskip0.2truecm\noindent
%iv) the weak closure $A''$ of $A$ in $\B(h)$ is amenable;
\vskip0.2truecm\noindent
iv) there exists a conditional expectation $E_\omega^L:\B(h)\to L^\infty (X,\mu_\omega)$ such that
\[
\Omega_L(T)=\int_X E^L_\omega (T)\cdot d\mu_\omega\qquad T\in\B(h).
\]
\end{prop}
\begin{proof}
Straitforward.
\end{proof}

\section{Appendix}
%The following extracts of J. Karamata theory are useful.
A measurable function $f:\R_+\to\R_+$ is said to be {\it regularly varying} if there exist the limits
\[
k(s):=\lim_{t\to +\infty}\frac{f(st)}{f(t)}\in (0,+\infty)\qquad \forall\,\, s>0.
\]
If $k(s)=1$ for all $s>0$, then $f$ is said to be {\it slowly varying}. Necessarily, $k$ must have the form $k(s)=s^\gamma$ for some $\gamma\in\mathbb{R}$, called the {\it index of regular variation} ($f\in R_\gamma$) and $f(t)=t^\gamma\ell(t)$ for some slowly varying function $\ell\in R_0$.
\begin{thm}(Karamata characterization)
The following characterization holds true:\\
$f\in R_\gamma$ if and only if for some $\sigma>-(\gamma+1)$ one has
\[
\lim_{t\to +\infty}\frac{t^{\gamma+1} f(t)}{\int_0^t x^\sigma f(x)dx}=\sigma+\gamma+1.
\]
\end{thm}
\begin{thm}(Karamata Tauberian Theorem)
Let $\mu$ be a positive Borel measure on $[0,+\infty)$ such that
\[
\int_0^{+\infty} e^{-tx}\mu(dx) <+\infty\qquad {\rm for\,\, all\,\,}t>0
\]
and suppose that it has a regularly varying Laplace transform (with index $\gamma\in\R$)
\[
\hat\mu(t):=\int_0^{+\infty} e^{-tx}\mu(dx)\qquad t>0.
\]
Then the function $N_\mu:\R_+\to\R_+$ defined as  $N_\mu (a):=\mu([0,a))$ has the following asymptotics
\[
N_\mu(a)=\mu([0,a))\sim\frac{{\hat\mu}(1/a)}{\Gamma(\gamma+1)}\qquad a\to+\infty .
\]
\end{thm}
Notice that the function $a\mapsto {\hat\mu}(1/a)$ is continuously differentiable as it is $\hat\mu$:
\[
\frac{d{\hat\mu}}{dt}(t)=-\int_0^{+\infty} xe^{-tx}\mu(dx)\qquad t>0.
\]

%\newpage

%------------------------------------------------------------------------------------------------------------------------------------------------
\end{document}